\documentclass[reqno,12pt]{amsart}

\usepackage[margin=2.4cm]{geometry}

\usepackage[ansinew]{inputenc}
\usepackage[english]{babel}
\usepackage[all]{xy}
\usepackage{amsmath,latexsym,amssymb,verbatim}
\usepackage{mathrsfs}
\usepackage[export]{adjustbox} 
\usepackage{ifthen}

\usepackage[usenames]{color}
\usepackage{hyperref}

\usepackage{amsfonts}

\usepackage{amsthm}
\usepackage{booktabs}
\usepackage{graphicx}
\usepackage{xcolor}
\usepackage[all]{xy}
\usepackage{hyperref}

\usepackage{mathtools}
\usepackage{subcaption}
\usepackage{faktor}
\usepackage{marginnote}
\usepackage{tikz-cd}
\usepackage{stmaryrd}
\usepackage{centernot} 
\usepackage{enumitem}

\newcommand{\rojo}[1]{\textcolor{red}{#1}}

\newcommand{\face}{{\trianglelefteqslant}}

\definecolor{azuldibujos}{RGB}{179, 230, 230}
\definecolor{contrast2}{HTML}{E50057}
\definecolor{contrast3}{HTML}{FFDD07}
\definecolor{H0}{HTML}{c8dfe3}
\definecolor{H1}{HTML}{5e8b94}

\newcommand\Z{\mathbb{Z}}

\newcommand\img{\operatorname{im}}

 \DeclareMathOperator\coker{coker}

\newcommand{\marginnotes}[1]{\ifthenelse{\isodd{\thepage}}{\normalmarginpar}
{\reversemarginpar}\marginpar{\fbox{\parbox{10mm}{\sloppy\footnotesize \rojo{#1}}}}}

\newtheorem{theorem}{Theorem}[section]
\newtheorem{proposition}[theorem]{Proposition}
\newtheorem{corollary}[theorem]{Corollary}

\theoremstyle{definition}
\newtheorem{definition}[theorem]{Definition}
\newtheorem{example}[theorem]{Example}
\AtBeginEnvironment{example}{
  \pushQED{\qed}
}
\AtEndEnvironment{example}{\popQED\endexample}
\newtheorem{remark}[theorem]{Remark}

\begin{document}

\title[]{From Persistence to Resilience: New Betti Numbers for Analyzing Robustness in Simplicial Complex Networks}

\date{\today}


\author[P. Hern\'andez-Garc\'ia]{Pablo Hern\'andez-Garc\'ia}
\author[D. Hern\'andez Serrano]{Daniel Hern\'andez Serrano}
\author[D. S\'anchez G\'omez]{Dar\'io S\'anchez G\'omez}

\address{Departamento de Matem\'aticas and Instituto Universitario de F\'isica Fundamental y Matem\'aticas (IUFFyM), Universidad de Salamanca, Salamanca, Spain}
\email{pablohg.eka@usal.es, dani@usal.es, dario@usal.es}

\begin{abstract}
    We show how persistence can be used to measure the robustness of cohomological cycles in finite simplicial complexes. We introduce two families of invariants: thick Betti numbers, which measure whether connected components and holes are supported by simplices of sufficiently high dimension; and cohesive Betti numbers, which measure the strength of higher-order adjacencies supporting cohomology classes. We then study robustness under simplicial degradation processes by combining an attack parameter with either the thick or the cohesive parameter, and by analyzing the resulting ladder modules through their image, kernel, and cokernel persistence modules. This allows us to distinguish features that remain structurally robust from those that lose thickness or cohesion during the degradation process. Finally, we prove stability results using the second network distance, which provides a theoretical reliability guarantee for the proposed constructions.
\end{abstract}

\maketitle
{\small \tableofcontents}

\section{Introduction}

Network Science \cite{Barabasi_Posfai16} provides a mathematical framework for studying complex systems modeled by graphs. This approach has led to a rich theory of connectivity, centrality, diffusion, community structure and robustness, with applications to social, biological, technological and physical networks. However, many real-world systems involve interactions among groups of agents, which has motivated the development of higher-order network models, where interactions are encoded by hypergraphs, simplicial complexes, cellular complexes or combinatorial complexes~\cite{Battiston_Otros20,Bick_Otros23}. Such structures have been used, among others, in neuroscience \cite{Giusti_Otros15,Reimann_Otros17}, in contagion dynamics~\cite{Barr21,Dani_Juan_Villarroel_Tocino23,Iacopini_Otros19}, in social systems~\cite{Atkin72,Freemam80,Dani_Juan_Dario20}, or topological deep learning~\cite{Haj23,Pa24}.

In Network Science, a system is robust if it preserves its basic functions during random failures or targeted attacks~\cite{Barabasi_Posfai16}. Traditionally, this property is assessed by observing changes in connectivity as nodes or edges are removed. However, when extending the study of robustness to simplicial complexes, the concept of structural damage expands, as degradation may involve the loss of simplices of any dimension. Indeed, topological percolation frameworks have shown that deleting higher-dimensional faces induces critical transitions that do not occur in ordinary graphs~\cite{Bianconi_Ziff18, Zhao_otros22_2, Zhao_Otros22}. Building on this, works such as~\cite{Bobrowski_Skraba20, Bobrowski_Skraba22} have extended percolation from simple connectivity to higher homological dimensions. From a complementary perspective, \cite{Chen_Otros23} pointed out that higher-order structures can actively increase robustness, providing ``synergistic protection'' to their lower-dimensional faces against failures. 

Despite these advances, evaluating higher-order robustness remains challenging with standard topological tools. While simplicial homology identifies global features such as connected components and higher-dimensional holes, Betti numbers are homotopy invariants and thus overlook internal combinatorial organization. For instance, two complexes may have the same number of connected components and holes, while differing in the dimension of the simplices supporting those holes or in the higher-order adjacencies that hold the corresponding cycles together. Thus, a proper assessment of robustness should track the ability of these homological features to survive structural degradation. This dynamic perspective is precisely what Topological Data Analysis offers.

Topological Data Analysis (TDA) uses tools from algebraic topology to study the evolution of homological features along filtrations of spaces. Persistent homology, originating in the works of \cite{Edelsbrunner_Letcher_Zomorodian02} and \cite{Carlsson_Zomorodian05}, formalizes this idea by assigning persistence modules and barcodes to filtered data. In this framework, Betti numbers provide interpretable summaries of the homological information at each stage and its practical impact is supported by efficient algorithms and specialized software \cite{Otter_Otros17,GUDHI,Giotto_tda,Ripser}, which have enabled applications in molecular and computational biology \cite{Kedi_Wei18,Xia_Wei_2014}, neuroscience \cite{Lee_Chung_Kang_Lee14,Petri_Otros14,Sizemore_Otros19}, oncology \cite{Masoomy_Otros21,Ramos_Otros25}, machine learning \cite{Hensel_Otros21,Hofer_17}, image processing \cite{Bleile_otros22,Singh_Otros23} and complex networks \cite{Ferra26, Horak09,Sizemore_Giusti_Bassett17}. These applications motivate studying not only the presence or persistence of homological features, but also how they are supported by the surrounding simplicial structure and how this support changes under degradation.

This is the point of departure of the present work. The perspective emphasized in~\cite{Ghrist08}, according to which ``a hole is a hole no matter how fragile or fine'', captures the topological nature of homological features and lies at the heart of persistent methods. Our viewpoint is complementary: in higher-order networks, fragility and fineness may themselves encode meaningful structural information. Thus, beyond tracking the birth and death of homological features, we study the structural quality of their support inside the simplicial complex.

The aim of this paper is to develop a persistence framework for finite simplicial complexes that refines Betti numbers by taking into account how cohomology classes are supported by the higher-order structure. For instance, a one-dimensional hole surrounded by filled triangles and one formed only by isolated edges may have the same Betti numbers, but they differ in the dimensional and adjacency structure that sustains them, and therefore should not be considered equally robust. This motivates the two complementary aspects studied below. The first one is \emph{thickness}: whether a connected component or a hole is supported by simplices of sufficiently high dimension. The second one is \emph{cohesiveness}: whether the simplices supporting a cohomology class remain strongly connected through higher-order adjacencies. 

Throughout the paper, \emph{robustness} refers to the preservation of topological features under prescribed structural modifications, such as removing simplices, restricting to higher-dimensional supports, or suppressing dimensional strata. By \emph{resilience}, we mean a dynamic persistent notion: the ability of these features to withstand a degradation process while retaining thickness or cohesiveness. Thus, resilience is not understood here as an autonomous adaptive response of the network, but as a persistent measure of its structural ability to resist possible attacks or failures. The invariants introduced below are designed to detect this form of structural robustness and may later inform application-specific adaptation or intervention strategies.

Our first construction, aimed at measuring thickness, is based on coskeleta. For each $q\geq 0$, we consider the subcomplex $X^q$ formed by faces of simplices of dimension at least $q$, obtaining a cofiltration that progressively removes simplices not supported by higher-dimensional cofaces. Taking cohomology leads to the \emph{thick Betti numbers}, which refine classical Betti numbers by recording the dimensional level at which connected components and holes remain visible.

Our second construction measures cohesiveness. Given a set of dimensions $\mathbf{h}$, we consider the subposet of the face poset of $X$ formed by simplices whose dimensions belong to $\mathbf{h}$. This leads to the \emph{cohesive Betti numbers}, which detect how cohomology classes depend on higher-order adjacencies among simplices of prescribed dimensions.

We then use these constructions to study simplicial degradation from a persistence perspective. Given a cofiltration modeling an attack, that is, the progressive removal of simplices, one can combine the attack parameter with either the thickness or cohesiveness parameter, obtaining a biparameter persistence module. Since such modules do not, in general, admit a barcode decomposition analogous to the one-parameter case~\cite{Carlsson_Zomorodian09}, we extract one-parameter information by fixing one structural parameter. This yields a ladder module~\cite{Escolar_Hiraoka16}, that is, a morphism between the persistence module based on the attack parameter and that of the corresponding thick or cohesive refinement. The image, kernel, and cokernel persistence modules~\cite{CohenSteiner_Otros09} allow us to distinguish features
that remain structurally robust from those that lose thickness or cohesion during
the degradation process.

A final issue concerns stability. Persistence invariants are expected to vary in a controlled way under small perturbations of the input data~\cite{Chazal_Otros09,Chazal_Otros16,Memoli_Chowdhury18,Steiner_Edelsbrunner_Harer07}. This property is essential in applications, where data are often affected by noise, measurement errors or incomplete sampling. We identify an appropriate metric framework for stability: while the classical network distance is too coarse to control the refined coskeletal and cohesive information, the second network distance~\cite{Chodhury_Memoli23} yields stability results for networks with the same number of vertices, providing theoretical support for future applications.

The paper is organized as follows. Section~\ref{Sec:Prelim} recalls the necessary background on simplicial complexes, cohomology, poset cohomology and persistence. Section~\ref{Sec:Coskeletal} introduces coskeleta (Definition~\ref{d:coske}) and thick Betti numbers (Definition~\ref{d:thickbetti}), proves their simplicial invariance (Theorem~\ref{t:thickbetti}), and studies their interpretation in degree zero (Proposition~\ref{Prop:Coskeletal0h}) and through coskeletal persistence (Theorem~\ref{Prop:InvarianzaFunctoresCoesq}, Proposition~\ref{Prop:aibiValues} and Corollary~\ref{Cor:StratificationCoskeletal}). Section~\ref{Sec:Stratified} defines cohesive Betti numbers (Definition~\ref{d:cohesivebetti}), proves their simplicial invariance (Theorem~\ref{t:cohesivebetti}), relates them to subcomplexes of the barycentric subdivision through equation~\eqref{e:BS}, and identifies their degree-zero interpretation in terms of higher-order connectivity (Proposition~\ref{Prop:StratifiedCaminos} and Corollary~\ref{cor:StratifiedCaminos}). Section~\ref{Sec:Resilience} applies the preceding constructions to cofiltrations modeling simplicial attacks, using the thick and cohesive ladder modules given by equations~\eqref{e:Thick6pack} and~\eqref{Eq:Diagram6pack} to analyze resilience via image, kernel, and cokernel persistence modules. Section~\ref{Sec:Stability} studies the stability of the proposed constructions for Vietoris--Rips cofiltrations, proving stability for thick and cohesive persistence modules and for their image, kernel, and cokernel modules (Theorems~\ref{Thm:EstabilidadGruesos},~\ref{Thm:EstabilidadPersistenciaGrosor},~\ref{Thm:EstabilidadCohesion} and~\ref{Thm:EstabilidadPersistenciaCohesion}). Finally, Appendix~\ref{app:demos} contains the proofs of several results stated throughout the paper.

\section{Preliminaries}\label{Sec:Prelim}

To start with, we will give a brief overview of some basic notions and facts about simplicial complexes, partially ordered sets (posets), and persistence which, in the sequel, will be useful for the convenience of the reader unfamiliar with these subjects. For a thorough exposition, we refer to~\cite{Hatcher02,Munkres84} for simplicial complexes and to~\cite{ Carlsson20,Chazal_Michel21, Ghrist18} for persistence. 

\subsection{Simplicial complexes and posets}\label{Subsec:SimplicialComplexes}

Given a finite set $V$, by a \emph{simplicial complex} on the ground set $V$ we mean a collection $X$ of non-empty subsets of $V$ closed under inclusion of subsets, that is, if $\sigma\in X$ and $\tau\subseteq \sigma$, then $\tau\in X$. The elements $\sigma$ of $X$ are called \emph{simplices} or \emph{faces} of $X$ and the elements of $\sigma$ are called \emph{vertices}. The vertices of $X$ are the one-point sets of $X$. We shall denote by $V(X)$ the set of vertices of $X$. If $\sigma\subseteq\tau$, we also say that $\sigma$ is a face of $\tau$ and we shall put $\sigma\face \tau$. A maximal face of $X$ (with respect to the inclusion of subsets) is called a \emph{facet}. Then, a simplicial complex is completely defined by the set of its facets. 

A simplex $\sigma\in X$ consisting of $n+1$ vertices is called an \emph{$n$-simplex} and we say that it has dimension $n$. The set of $n$-simplices of a simplicial complex $X$ is denoted by $S^{n}(X)$. The \emph{dimension} of a simplicial complex $X$ is defined as the maximum of the dimensions of its simplices, $\dim\,X= \max\{\dim\,\sigma : \sigma \in X\}$. 

A  \emph{simplicial map} $f\colon X\to Y$ between two simplicial complexes $X$ and $Y$ is a map $f\colon V(X)\to V(Y)$ between its vertices such that, for any 
simplex $\sigma$ in $X$, the subset $f(\sigma)\coloneqq \{f(v)\,\colon\, v\in \sigma\}$ is a simplex in $Y$. Note that injective maps preserve the dimension of the simplices but, in general, this dimension may be smaller. A simplicial map $f\colon X\to Y$ is called a \emph{simplicial isomorphism} if $f\colon V(X)\to V(Y)$ is bijective and its inverse  $f^{-1}$ is a  simplicial map of $Y$ into $X$. 

The \emph{face poset} of a simplicial complex \(X\), denoted by \( P_X \), is defined as the set of simplices of \( X \) equipped with the partial order given by the face relation \( \face \). On the other hand, given a poset $P$ we shall denote by $\mathcal{K}(P)$ its \emph{order complex}, that is, the simplicial complex whose simplices are the totally ordered sets (i.e., chains) of points of $P$. In particular, when the poset is a face poset, the simplicial complex formed by the chains in the face poset is known to be the \emph{barycentric subdivision} of the original simplicial complex $X$. It shall be denoted as $\mathcal{K}(X)$.

From now on, we will always assume that $V(X)$ is a totally ordered set and $n$-simplices will be denoted as $\sigma=({i_0},\dots,{i_n})$, where $i_0<i_1<\cdots <i_n$ are nonnegative integers. 

Given a simplicial complex $X$ and a field $\Bbbk$, the space of $n$-cochains on $X$, $C^n(X;\Bbbk)$, is the $\Bbbk$-vector space spanned by the set of $n$-simplices of $X$. For an $n$-cochain $x\in C^n(X;\Bbbk)$, we denote by $x_\sigma$ its component on the simplex $\sigma\in S^n(X)$. After fixing a total ordering of the vertices of $X$, the $n$-th simplicial coboundary operator $\delta^n\colon C^n(X;\Bbbk)\to C^{n+1}(X;\Bbbk)$ is given, for each $\sigma=(i_0,\dots,i_{n+1})$, by
\[
(\delta^n x)_\sigma=\sum_{j=0}^{n+1}(-1)^j x_{(i_0,\dots,\widehat{i_j},\dots,i_{n+1})},
\]
where $\widehat{i_j}$ means that $i_j$ is omitted. These maps satisfy $\delta^{n+1}\circ\delta^n=0$, and hence define the simplicial cochain complex of $X$. Its $n$-th cohomology space is
\[
H^n(X;\Bbbk)=\ker\,\delta^n/\img\,\delta^{n-1}.
\] 
A straightforward computation shows that these spaces do not depend on the total order fixed on the set of vertices. The dimension of the $n$-th simplicial cohomology space associated to $X$ with coefficients in $\Bbbk$ is known as the \emph{$n$-th Betti number} of $X$ with coefficients in $\Bbbk$ and denoted by $\beta^n(X;\Bbbk)$.

\begin{remark}\label{rem:simp-coh}
The usual definition of simplicial cohomology~\cite{Gallier_Quaintance22} is based on the cochain spaces and coboundary maps obtained by taking the dual of the chain spaces and boundary maps that define the simplicial homology~\cite{Munkres84}. By finite-dimensional duality, the definition presented here is equivalent to the usual one. On the other hand, Betti numbers of simplicial complexes are usually defined as the dimension of simplicial homology spaces rather than the dimension of simplicial cohomology spaces. However, for finite simplicial complexes, both definitions agree, as follows from the Universal Coefficient Theorem for cohomology~\cite[Proposition 12.9]{Gallier_Quaintance22}.
\end{remark}

Every simplicial map $f\colon X\to Y$ induces linear maps ${f^n}\colon C^n(Y;\Bbbk)\to C^n(X;\Bbbk)$ satisfying ${f^{n+1}}\circ \delta^n=\delta^n\circ {f^n}$ (see Appendix~\ref{app:demos}), so they induce morphisms between cohomology spaces
\[
H^n(f)\colon  H^n(Y;\Bbbk)\longrightarrow  H^n(X;\Bbbk)\,.
\]
These assignments are functorial and, as a consequence, Betti numbers are simplicial invariants. Explicitly:

\begin{proposition}\label{Prop:BettiInvariantes}
If $X$ and $Y$ are isomorphic finite simplicial complexes, then $\beta^n(X;\Bbbk)=\beta^n(Y;\Bbbk)$ for every $n\geq 0$. 
\end{proposition}

Recall also that Betti numbers give us topological information about the structure of the simplicial complex: $\beta^0(X;\Bbbk)$ counts the number of connected components of $X$ and, for $n\geq 1$, $\beta^n(X;\Bbbk)$ is usually interpreted as the number of ``$n$-dimensional holes'' in the simplicial complex.

\begin{definition}
The $n$-th cohomology space of a poset $P$ with coefficients in $\Bbbk$ is defined as the simplicial cohomology of its order complex, that is,
$$H^n(P;\Bbbk)\coloneqq H^n(\mathcal{K}(P); \Bbbk)\,.$$
\end{definition}

Any monotone map $f\colon P\to Q$ between posets induces a simplicial map between the corresponding order complexes $f\colon \mathcal{K}(P)\to\mathcal{K}(Q)$. Then, $f$ also induces (in a functorial way) linear maps in cohomology $H^n(f)\colon  H^n(Q;\Bbbk)\longrightarrow  H^n(P;\Bbbk)$.

When the poset is a face poset, its order complex is exactly the barycentric subdivision of the original simplicial complex. Consequently, as follows from~\cite[Theorem~17.2]{Munkres84} and Remark~\ref{rem:simp-coh}, both structures yield isomorphic cohomology spaces.
\begin{theorem}
    The $n$-th simplicial cohomology with coefficients in $\Bbbk$ of a simplicial complex $X$ is isomorphic to the $n$-th poset cohomology with coefficients in $\Bbbk$ of its face poset $P_X$:
    \begin{equation*}
        H^n(X;\Bbbk)\simeq H^n(P_X;\Bbbk).
    \end{equation*}
\end{theorem}

\begin{corollary}\label{Cor:BettiNumbersSheafCohomology}
    The Betti numbers of a simplicial complex $X$ can be computed via poset cohomology as
    \begin{equation*}
        \beta^{n}(X;\Bbbk)=\dim H^n(P_X;\Bbbk).
    \end{equation*}
\end{corollary}

Geometrically, a simplicial complex can be understood by gluing simplices together by means of their faces. Hence, we can consider different degrees of adjacency between the simplices. Here, we focus on the notion of lower adjacencies introduced in~\cite{Dani_Juan_Dario20}. 
\begin{definition}
Let $\tau$ and $\tau'$ be simplices of a simplicial complex $X$. We say that $\tau$ and $\tau'$ are \emph{$h$-lower adjacent} if there exist a $h$-simplex $\sigma$ such that $\sigma\face \tau$ and $\sigma\face \tau'$. In this case, we write $\tau\sim_{L_h} \tau'$.
\end{definition}

\begin{definition}
Let $X$ be a simplicial complex and let $0\leq h_0< h_1$. 
\begin{itemize}
    \item An \emph{$(h_0,h_1)$-walk} is a sequence $\{\tau_1,\tau_2,\ldots,\tau_{r}\}$ of $h_0$-lower adjacent simplices such that $h_1=\min\{\dim\tau_i\}$, that is,
    \begin{equation*}
    \tau_1\sim_{L_{h_0}}\tau_2\sim_{L_{h_0}}\cdots\sim_{L_{h_0}} \tau_r \, ,    
           \quad h_1=\min\,\dim \,\tau_i.
    \end{equation*}
    We say that a walk is closed if $\tau_1=\tau_r$.
    
    \item For a pair of simplices $\sigma$ and $\sigma'$ we define 
    \begin{equation}\label{Eq:HigherOrderAdjacency}
    \sigma\sim_{h_0h_1}\sigma' \iff \exists\quad  (h_0,h_1)\text{-walk }\, \{\tau_1,\tau_2,\ldots,\tau_r\}  : \sigma\face \tau_1\, \text{ and }\,  \sigma'\face \tau_r\,.
    \end{equation}
    In addition, we set that $\sigma\sim_{h_0h_1}\sigma$ for each $\sigma\in X$. 

    \item We call \emph{$(h_0,h_1)$-connected components} of $X$ the equivalence classes of the quotient set ${S^{h_0}(X)}/{\sim_{h_0h_1}}$. 
\end{itemize}
\end{definition}

\begin{example}
   Figure~\ref{Fig:Componentes2Conexas} shows a connected simplicial complex with three $(0,2)$-connected components (marked in red). The $(0,2)$-connected components measure the connectivity of the vertices through walks of adjacent triangles. 
\end{example}
\begin{figure}[!htb]
\centering
\includegraphics[scale=0.6]{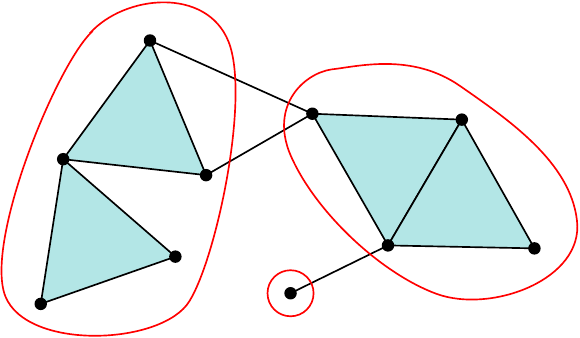}
\caption{$(0,2)$-connected components}\label{Fig:Componentes2Conexas}
\end{figure}

\begin{remark}\label{rem:qconnectivity}
This connectivity is closely related to the ``$q$-connectivity'' considered, for instance, in~\cite{Atkin72,  CR24,Dani_Dario20, Maletic_Rajkovic12, Riihimaki23}. For example, the set of $q$-connected components defined in~\cite{Maletic_Rajkovic12} is in bijection with the set of $(q,q+1)$-connected components whose equivalence classes contain more than one element. In particular, the $(0,1)$-connected components are the usual connected components of the simplicial complex. 
\end{remark}

\subsection{Persistence} \label{Subsec:Persistence} \quad

The idea behind persistent homology is to track the topological features of simplicial complexes along a filtration by means of homology. It is primarily applied to study the shape of point clouds~\cite{ Carlsson20,Chazal_Michel21, Ghrist18}, but is not limited to this setting. In Network Science, persistent homology has been used to analyze complex and weighted networks through graph filtrations, clique complexes, and mesoscale homological features~\cite{Bergomi_Ferri_Zuffi20,Ferra26,Horak09,Petri_Otros14,Sizemore_Giusti_Bassett17}. It has also been developed in other mathematical contexts, including geometry and analysis~\cite{Polterovich_Rosen_Samvelyan_Zhang20} and persistent local systems over manifolds~\cite{MacPherson_Patel21}. In this paper, persistence is used in this general sense: as a way of organizing the evolution of cohomological information along structured families of spaces. Since our constructions are based on simplicial removal processes, we introduce persistence from a dual point of view, using cofiltrations and simplicial cohomology.

Let $P$ be a poset and let $\{X^p\}_{p\in P}$ be a \emph{cofiltration} of simplicial complexes indexed by $P$, that is, a family of simplicial complexes such that $X^p\supseteq X^q$ whenever $p\leq q$. As we have mentioned earlier, the aim is to follow the changes in the topology of the simplicial complexes along the cofiltration. This can be performed by studying the vector spaces $H^n(X^p;\Bbbk)$ and the linear maps $\varphi_{pq}\colon H^n(X^p;\Bbbk)\to H^n(X^q;\Bbbk)$ induced by the inclusions $X^p\supseteq X^q$. The resulting object is known as a \emph{persistence module}, a notion which appears in~\cite{Carlsson_Otros04} motivated by the study of the homology of a filtration.
 
Let us recall that a \emph{persistence module} $\mathcal{M}$ indexed by a poset $P$ consists of a family of $\Bbbk$-vector spaces $\{M_p\}_{p\in P}$ and a family of linear maps $\{\varphi_{pq}\colon \mathcal{M}_p\to \mathcal{M}_q\}_{p\leq q}$, such that $\varphi_{pp}=\mathrm{Id}_{\mathcal{M}_p}$ for all $p\in P$, and $\varphi_{qr}\circ \varphi_{pq}=\varphi_{pr}$ for all $p\leq q \leq r$. 
We say that the persistence module $\mathcal{M}$ is \emph{pointwise finite-dimensional} (abbreviated as \emph{p.f.d.}) if all the vector spaces $\mathcal{M}_p$ are finite-dimensional. A \emph{morphism of persistence modules} $\Phi\colon \mathcal{M}\to \mathcal{N}$ is a set of linear maps $\{\Phi_p\colon \mathcal{M}_p\to \mathcal{N}_p\}_{p\in P}$ 
that satisfy, for each $p\leq q$, $\varphi'_{pq}\circ\Phi_p=\Phi_q\circ\varphi_{pq}$. If $\Phi_p$ is bijective for all $p\in P$, then $\Phi$ is an isomorphism.

Persistence modules indexed by $\mathbb{N}^r$ are called $r$-\emph{parameter} (discrete) persistence modules. In the one-parameter case, a fundamental class of persistence modules consists of the so-called \emph{interval persistence modules}. Given a totally ordered set $T$ and an interval $I\subseteq T$, that is, a subset such that if  $t\leq s\leq r$ and $t,r\in I$, then $s\in I$ as well, the \emph{interval persistence module} $\Bbbk[I]$ associated to $I$ is defined as
\begin{equation*}
\Bbbk[I]_t=\begin{cases}\Bbbk & \text{ if } t\in I ,\\
0 & \text{ otherwise,}
\end{cases}
\ \ \ \text{ and } \ \ \ 
\varphi_{ts}=\begin{cases}
\mathrm{Id}_\Bbbk & \text{ if } t,s\in I, \\
0 & \text{ otherwise.}
\end{cases}
\end{equation*}
 
The representation theorem for one-parameter discrete persistence modules~\cite{Carlsson_Zomorodian05} shows that interval persistence modules are the basic building blocks to obtain other persistence modules. A generalization of this structure theorem was given by~\cite{Crawley-Boevey15} as follows.
\begin{theorem}\label{Thm:Descomposicion} 
Every pointwise finite-dimensional persistence module $\mathcal{M}$ indexed by a totally ordered set decomposes in a unique way (up to reordering) as a direct sum of interval modules, $
\mathcal{M}\simeq \displaystyle\bigoplus_{\lambda\in \Lambda}\Bbbk[I_\lambda]
$.
\end{theorem}
The collection of intervals $\{I_\lambda\}_{\lambda\in \Lambda}$ for which $\mathcal{M}\simeq \bigoplus_{\lambda\in \Lambda}\Bbbk[I_\lambda]$ is known as the \emph{persistence barcode} of $\mathcal{M}$. From a computational standpoint, this theoretical construct is highly accessible, with several software implementations available to extract persistence barcodes from data (see~\cite{Otter_Otros17}). 

\begin{remark}
Notice that the total ordering of the index set is essential; for example, in general, for persistence modules indexed by $\mathbb{N}\times\mathbb{N}$ there is no analogous decomposition result.  
\end{remark}

\begin{example}
Consider a finite one-parameter cofiltration of simplicial complexes    
\begin{equation*}
        X^0\supseteq X^1\supseteq X^2\supseteq \dots \supseteq X^M,
    \end{equation*}
arising, for instance, from the successive removal of simplices from an initial simplicial complex $X^0$, which models an arbitrary degenerative process. Taking the $n$-th simplicial cohomology one gets a persistence module 
    \begin{equation*}
        H^n(X^\bullet;\Bbbk)\equiv H^n(X^0;\Bbbk)\to H^n(X^1;\Bbbk)\to \dots \to H^n(X^M;\Bbbk).
    \end{equation*}    
    By Theorem~\ref{Thm:Descomposicion}, it decomposes as a direct sum of intervals of the finite set $\{0,1,\dots, M\}$, 
    \begin{equation*}
        H^n(X^\bullet;\Bbbk)\simeq \bigoplus_{i=1}^m\Bbbk[b_i,d_i) \quad \text{for some }0\leq b_i<d_i\leq M+1.
    \end{equation*}
    The persistence barcode, $\{[b_i,d_i)\}_{i=1}^m$, is represented as a set of bars over the real line. These bars visually encode the behaviour of the topological features along the cofiltration  (Figure~\ref{Fig:FiltracionDiscretaCodigoBarras}). Each bar corresponds to an interval during which a cohomology class persists, with $b_i$ and $d_i$ indicating its birth and death, respectively. In particular, long bars correspond to cohomology classes with high persistence, which are therefore more robust to the deletion of simplices from the initial simplicial complex $X^0$. 
\end{example}

\begin{example}
Figure~\ref{Fig:FiltracionDiscretaCodigoBarras} provides an example of a cofiltration indexed by the ordered set $\{0,1,2,3,4,5\}$ and the barcodes of the persistence modules associated with the $0$-th and $1$-st simplicial cohomology with coefficients in $\Z_2$. The $H^0(X^\bullet;\Z_2)$ barcode tracks the connected components of the simplicial complexes. A bar starts at $0$ and reflects the connectivity of the initial simplicial complex. The connectivity is not affected by the first two removals, but it is at step $3$, where one of the vertices is disconnected from the rest, so a new bar appears. This bar persists until step $5$, when the vertex on the right is removed. Also in step $5$, two connected components are born due to the elimination of the remaining edges. As for the $1$-dimensional holes captured by $H^1(X^\bullet;\Z_2)$, there are two bars in the corresponding barcode. Both bars start at step $1$, when two holes are ``born'' due to the removal of the interior of the triangles from the initial simplicial complex. One of them ``dies'' immediately in step $2$ but the other persists until step $4$, as shown by the top bar.
\end{example}

\begin{figure}[!htb]
       \centering
       \includegraphics[width=0.85\linewidth]{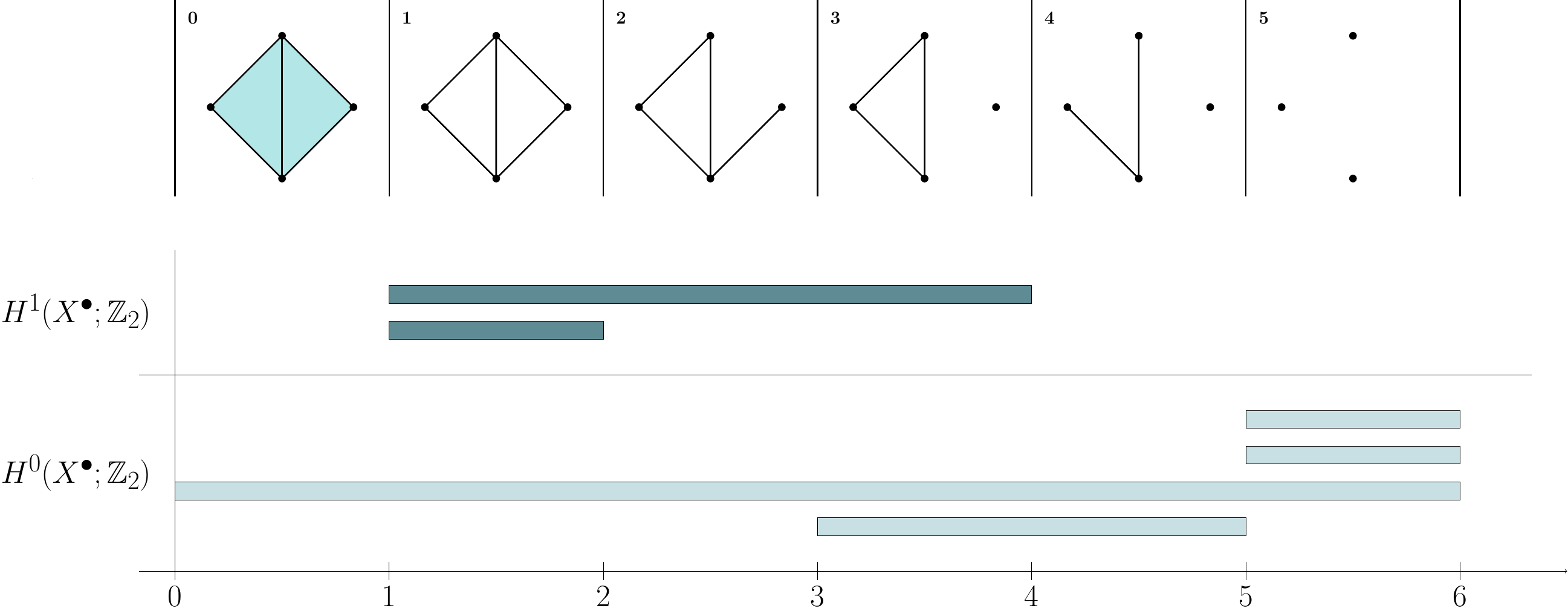}
\caption{Cofiltration of a simplicial complex and the barcodes of the persistence modules associated with the $0$-th and $1$-st simplicial cohomology with coefficients in $\Z_2$}\label{Fig:FiltracionDiscretaCodigoBarras}
\end{figure}

Finally, let us mention that unlike the 1-parameter case, the analysis of multi-parameter persistence modules presents several challenges~\cite{Botnan_Lesnick23}. A fundamental difficulty is that a complete discrete invariant, such as the barcode, does not generally exist in the multidimensional setting~\cite{Carlsson_Zomorodian09}. To address this, several invariants have been proposed in TDA for multiparameter persistence, such as the Hilbert function, the rank invariant, or the bigraded Betti numbers. Moreover, persistence barcodes may be obtained from a biparameter persistence module by restricting it to persistence modules defined by totally ordered subsets of $\mathbb{N}\times \mathbb{N}$, such as its vertical or horizontal components. To facilitate the computation and visualization of these invariants, software tools such as RIVET~\cite{RIVET, Lesnick_Wright15} have been developed.

\section{How thick is a hole?}\label{Sec:Coskeletal}

In this section, we will use coskeletons and persistence to introduce a collection of new invariants associated to (the category of)  simplicial complexes: the \emph{thick Betti numbers}. As we shall see, these invariants reveal structural properties of the cohomological cycles of a simplicial complex related to the dimension of their constituting simplices, providing us a way to quantify the robustness of a simplicial complex in terms of the thickness of its connected components and holes.

\subsection{Thick Betti numbers}

\begin{definition}\label{d:coske}
Given a simplicial complex $X$ and $q\in \Z_{\geq 0}$, the \emph{$q$-coskeleton} of $X$ is defined as 
    \begin{equation*}
        X^q=\{\sigma\in X : \exists\, \tau\in X \, \text{ such that }\, \sigma\face\tau \, \text{ and } \dim \tau\geq q\}\,.
    \end{equation*}
\end{definition}

The $q$-coskeleton is a  simplicial subcomplex of $X$ and considers the simplices of dimension greater than or equal to $q$ together with all their faces (Figure~\ref{Fig:EjemploCoskeletons}). In other words, the $1$-coskeleton removes all vertices that do not belong to any edge, the $2$-coskeleton removes all vertices and edges that are not part of any filled triangle, and so on. Thus, as $q$ increases, the coskeleta progressively isolate the thicker part of the complex, namely the topology supported by higher-dimensional simplices. The comparison with $X$ reveals whether the removal of smaller facets changes the topology in a significant way.

\begin{figure}[!htb]
\centering
\begin{subfigure}{0.47\textwidth}
\centering
    \includegraphics[scale=0.3]{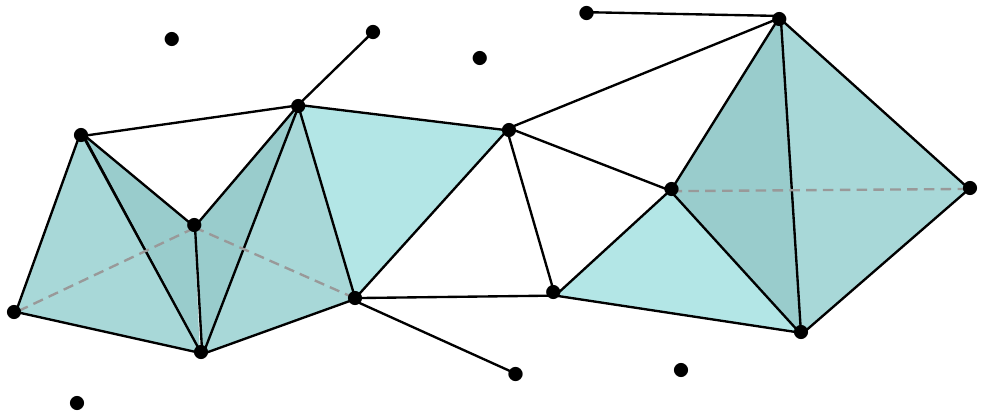}
    \caption{Simplicial complex $X$}\label{Fig:Ejemplo3PXh}
\end{subfigure}
\hfill
\begin{subfigure}{0.47\textwidth}
\centering
    \includegraphics[scale=0.3]{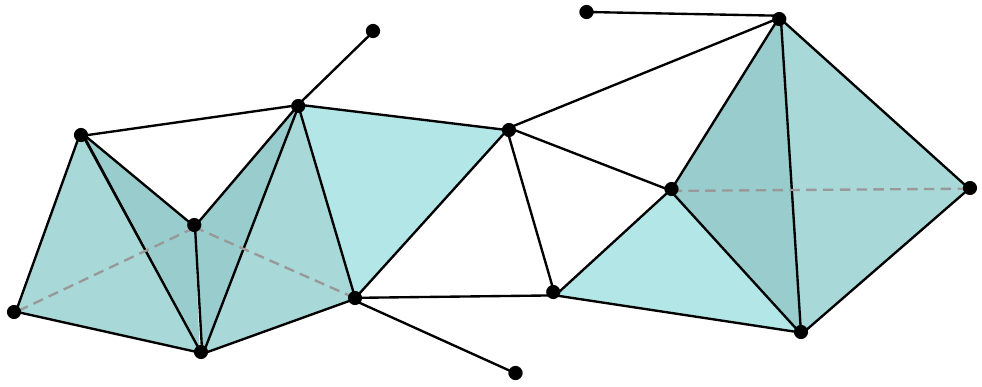}
    \caption{$1$-coskeleton $X^1$}\label{Fig:Ejemplo3PX1}
\end{subfigure}

\vspace{1.5em} 

\begin{subfigure}{0.47\textwidth}
\centering
    \includegraphics[scale=0.3]{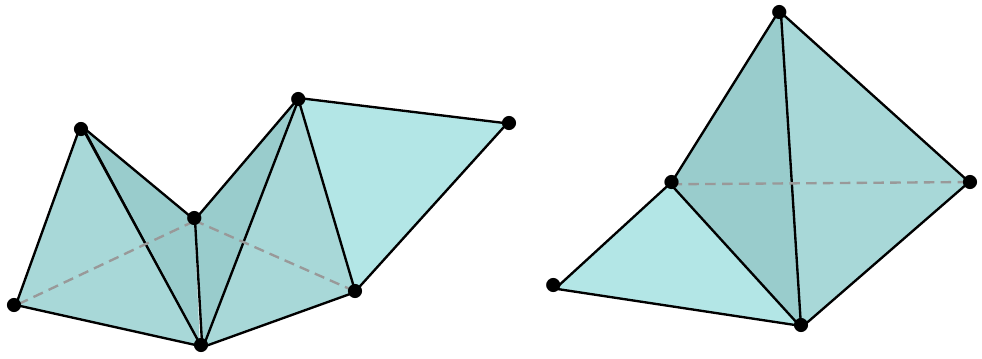}
    \caption{$2$-coskeleton $X^2$}\label{Fig:Ejemplo3PX2}
\end{subfigure}
\hfill 
\begin{subfigure}{0.47\textwidth}
\centering
    \includegraphics[scale=0.3]{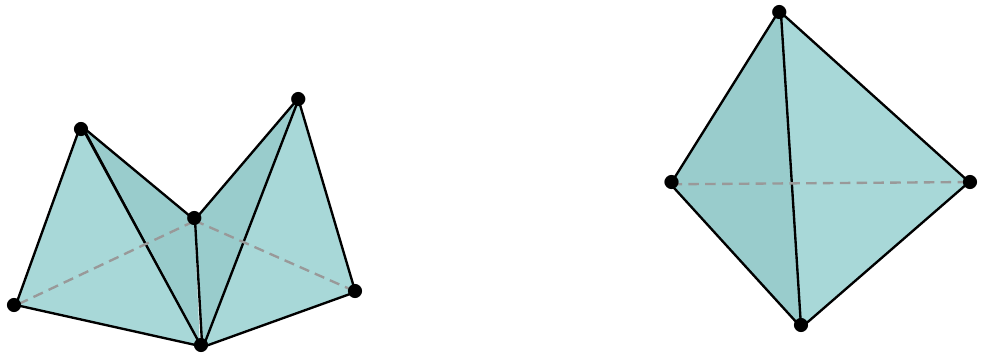}
    \caption{$3$-coskeleton $X^3$}\label{Fig:Ejemplo3PX3}
\end{subfigure}
\caption{We can recursively obtain the $q$-coskeleton by eliminating from the $(q-1)$-coskeleton the simplices that are not a face of a $q$-simplex}\label{Fig:EjemploCoskeletons}
\end{figure}
\begin{definition}\label{d:thickbetti}
The \emph{$(n,q)$-th thick Betti number} of a simplicial complex $X$ with coefficients in $\Bbbk$ is the dimension of the $n$-th simplicial cohomology space with coefficients in $\Bbbk$ of the $q$-coskeleton of $X$, that is,
\begin{equation*}
\beta^{n,q}(X;\Bbbk)=\dim H^n(X^q;\Bbbk)\,.
\end{equation*}
\end{definition}
The  $(n,q)$-th thick Betti number can be interpreted as the number of $n$-dimensional cycles which are surrounded by simplices of dimension at least $q$. In particular, $\beta^{n,0}(X;\Bbbk)=\beta^n(X;\Bbbk)$ for all $n\geq 0$, so the set of thick Betti numbers extends the set of Betti numbers.

Note that every simplicial isomorphism $f\colon X \stackrel{\sim}\to Y$ yields, by restriction, an isomorphism between the \mbox{$q$-coskeletons}, $f^q\colon X^q\stackrel{\sim}\to Y^q$. Thus, by  Proposition~\ref{Prop:BettiInvariantes}, one gets that $\dim H^n(X^q;\Bbbk)=\dim H^n(Y^q;\Bbbk)$, for all $n\geq 0$. Hence, as with the usual Betti numbers, the thick Betti numbers are simplicial invariants as well.
\begin{theorem}\label{t:thickbetti}
If $X$ and $Y$ are isomorphic simplicial complexes, then $\beta^{n,q}(X;\Bbbk)=\beta^{n,q}(Y;\Bbbk)$ for all $n,q\geq 0$.
\end{theorem}

\begin{example}
In Figure~\ref{Fig:ExCosk23Triang}a, we have a $1$-dimensional hole in $X$ (empty triangle, $\beta^{1}(X;\mathbb{Z}_2)=1$) that is fully enclosed by filled triangles ($2$-simplices). Taking the $2$-coskeleton of $X$ does not alter $X$, and the $(1,2)$-thick Betti number, $\beta^{1,2}(X;\mathbb{Z}_2)=1$, confirms that the hole persists as shown in (b). In (c), the $1$-dimensional hole in $Y$ (empty triangle, $\beta^{1}(Y;\mathbb{Z}_2)=1$) is not completely surrounded by filled triangles, causing the hole to disappear when taking the $2$-coskeleton. This is captured by the thick Betti number, $\beta^{1,2}(Y;\mathbb{Z}_2)=0$, as shown in (d).  Therefore, one can say that the hole in $X$ is thicker than the hole in $Y$, as it is fully enclosed by $2$-simplices.
\end{example}
\begin{figure}[htb!]
\centering
\begin{minipage}{\textwidth}
    \centering
    \begin{minipage}{0.35\textwidth}
        \centering 
        \includegraphics[scale=0.6]{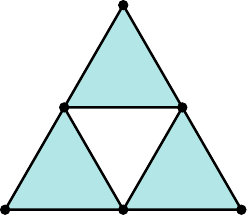}
        \subcaption{$\beta^1(X;\Z_2)=1$}\label{Fig:ExCosk23TriangA}
    \end{minipage}
    \begin{minipage}{0.09\textwidth}
        \centering
        \vspace{-7ex}
        $$\stackrel{q=2}{\scalebox{2.2}[1]{$\longmapsto$}}$$
    \end{minipage}
    \begin{minipage}{0.35\textwidth}
        \centering
        \includegraphics[scale=0.6]{ExCosk3Triang.pdf}
        \subcaption{$\beta^{1,2}(X;\Z_2)=1$}
    \end{minipage}
\end{minipage}

\vspace{1em}

\begin{minipage}{\textwidth}
    \centering
    \begin{minipage}{0.35\textwidth}
        \centering
        \includegraphics[scale=0.6]{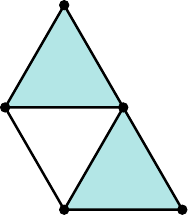}
        \subcaption{$\beta^1(Y;\Z_2)=1$}
    \end{minipage}
    \begin{minipage}{0.09\textwidth}
        \centering
        \vspace{-7ex}
        $$\stackrel{q=2}{\scalebox{2.2}[1]{$\longmapsto$}}$$
    \end{minipage}
    \begin{minipage}{0.35\textwidth}
        \centering 
        \includegraphics[scale=0.6]{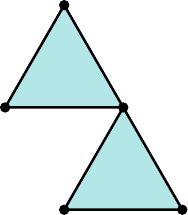}
        \subcaption{$\beta^{1,2}(Y;\Z_2)=0$}
    \end{minipage}
\end{minipage}
\caption{Thickness of a hole: a $1$-dimensional hole fully enclosed by $2$-simplices in (a) and a hole not surrounded by $2$-simplices in (c). This structural difference is captured by considering $2$-coskeletons: the hole does not change in $X$, figure (b), whereas the hole in $Y$ disappears, figure (d)}
\label{Fig:ExCosk23Triang}
\end{figure}

For a preliminary definition, we will consider an $n$-dimensional cycle to be  \emph{thicker} than other if it is enclosed by higher dimensional simplices. However, as we will explore in the next section, this requires a more formal treatment.

The following proposition shows that the $(0,q)$-th thick Betti number counts the number of connected components which are defined by lower adjacent $q$-dimensional simplices (see Figure~\ref{Fig:0hCoskeletal}). 
\begin{figure}[!htb]
\centering
\begin{subfigure}{0.45\textwidth}
\centering
\includegraphics[scale=0.45]{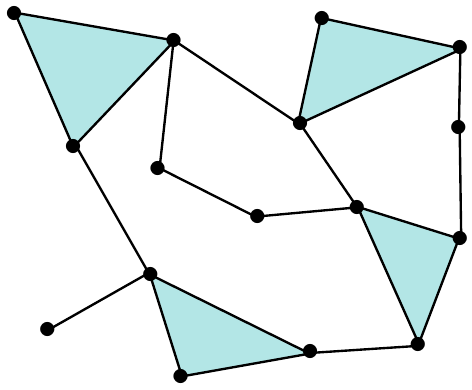}
\caption*{$\beta^{0,2}(X;\Z_2)=4$}
\end{subfigure}
\begin{subfigure}{0.45\textwidth}
\centering
\includegraphics[scale=0.45]{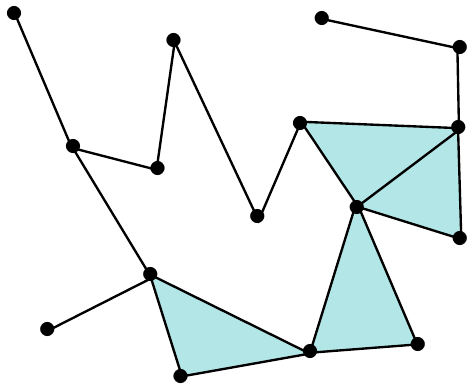}
\caption*{$\beta^{0,2}(Y;\Z_2)=1$}
\end{subfigure}
\caption{Two connected simplicial complexes. The $(0,2)$-th thick Betti numbers measure the connectivity between the triangles}\label{Fig:0hCoskeletal}
\end{figure}
\begin{proposition}\label{Prop:Coskeletal0h}
For any $q>0$ and any field $\Bbbk$, the following formula holds:
    \begin{equation*}
        \beta^{0,q}(X;\Bbbk)=\#\faktor{V(X)}{\sim_{0q}}-\#\{v\in V(X) :  v\notin\sigma \textup{ for all }  \sigma\in S^{q}(X)\}\, ,
    \end{equation*}
or, equivalently, 
\begin{equation*}
\beta^{0,q}(X;\Bbbk)=\#\faktor{S^{\geq q}(X)}{\sim_{0q}}\, ,
\end{equation*}
where $S^{\geq q}(X)$ denotes the set of simplices of $X$ of dimension at least $q$.
\end{proposition}
    \begin{proof}
See Appendix~\ref{proof:Prop-Coskeletal0h} for details.
 \end{proof}
 
\begin{remark}
Similar strategies isolating higher-dimensional structures can be found in the literature. For instance, \cite{ Connon_Faridi_2015,DKM09} study the pure $q$-skeleton of a complex, which is the simplicial subcomplex generated by all $q$-simplices (and thus a subcomplex of our $q$-coskeleton). Along a similar line, \cite{CR24} focus on simplices of dimension greater than or equal to $q$, but shift the framework from simplicial complexes to digraphs to analyze $q$-connectivity (Remark~\ref{rem:qconnectivity}).
\end{remark}

\subsection{Persistence of a hole's thickness}\label{Subsec:ThicknessPersistence}

Now we would like to answer the following question: how robust is a hole under the simplicial elimination rule defined by taking coskeletons?

We mentioned earlier that a hole could be considered thicker than another if it was enclosed by higher-dimensional simplices. However, as illustrated in Figure~\ref{Fig:ExCosk23Triang}, thick Betti numbers do not provide a static way to verify this but instead require a dynamic perspective. It is through the successive simplicial removals that we can determine whether a ``thick'' hole remains or disappears. This is exactly what the persistence telescope helps us achieve.

The coskeletons of a simplicial complex $X$ give rise to a cofiltration: 
\begin{equation*}
X=X^0\supseteq X^1\supseteq X^2\supseteq \dots \supseteq X^{N}\supseteq X^{N+1}=\emptyset\,,\quad N=\dim X.
\end{equation*}
Taking cohomology in degree $n\geq 0$ in the previous cofiltration, we obtain a sequence of $\Bbbk$-vector spaces and linear maps:
\begin{equation}\label{Eq:SucesionCohomologiaCoesqueletos}
H^n(X^0;\Bbbk)\to H^n(X^1;\Bbbk)\to \dots \to H^n(X^{N-1};\Bbbk)\to H^n(X^N;\Bbbk)\to 0.
\end{equation}
In short, by simplicial cohomology functoriality, the coskeletons induce, for each $n\geq 0$, a persistence module $H^n(X^\bullet;\Bbbk)$ which captures how cohomology classes evolve along the cofiltration. 

In particular, since $X^0=X$, we can analyze the structure of connected components and holes in $X$ by tracking their persistence as successive coskeletons are taken. This provides a robustness measure for the holes in a simplicial complex: the cohomology classes in $H^n(X;\Bbbk)$ with longer persistence correspond to those defined by higher-dimensional facets, meaning they are thicker.
 
\begin{remark}
It is also worth noting that if we intend to rigorously analyze other events occurring throughout the coskeletal cofiltration, such as the appearance of holes or connected components, persistence provides the most appropriate framework for this study, as shown in Figure~\ref{Fig:EjemploCreacionAgujero} (birth of a thick hole) or illustrated in Example~\ref{ex:splittingconncomp} (splitting into new connected components).
\end{remark}

\begin{figure}[!htb]
\centering
    \begin{minipage}{0.4\textwidth}
        \centering
        \includegraphics[scale=0.4]{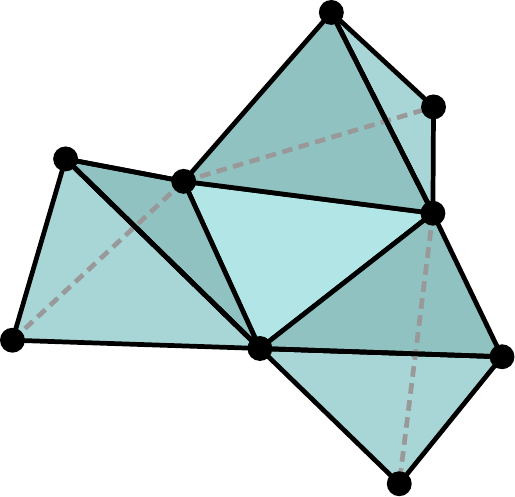}
        \subcaption{$\beta^1(X;\Z_2)=0$}
    \end{minipage}
    \begin{minipage}{0.09\textwidth}
        \centering
        \vspace{-7ex}
        $$\stackrel{q=3}{\scalebox{2.2}[1]{$\longmapsto$}}$$
    \end{minipage}
    \begin{minipage}{0.4\textwidth}
        \centering
        \includegraphics[scale=0.4]{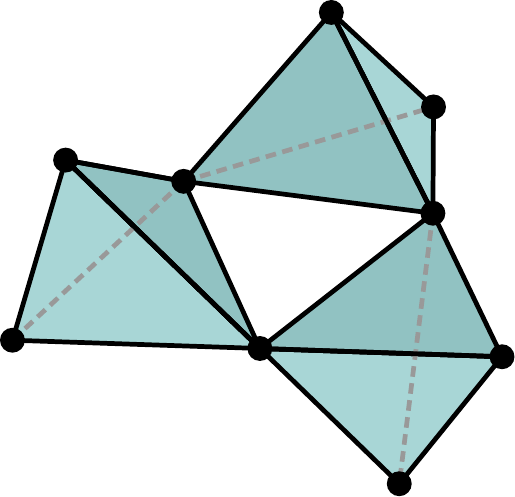}
        \subcaption{$\beta^{1,3}(X;\Z_2)=1$}
    \end{minipage}
\caption{Example of the birth of a hole in the coskeletal cofiltration}\label{Fig:EjemploCreacionAgujero}
\end{figure}

\begin{example}\label{ex:splittingconncomp}
The simplicial complex in Figure~\ref{Fig:RepresentacionCodigoBarras}a has two connected components, the same as its non-trivial coskeletons. However, when we take the $2$-coskeleton (Figure~\ref{Fig:RepresentacionCodigoBarras}b), one connected component disappears and the other one is split into two. So, although the Betti numbers remain constant along the coskeletal cofiltration, there may be structural changes when taking successive coskeletons. In (c), the barcode of the persistence module $H^0(X^\bullet;\Z_2)$ allows us to capture these changes: when $q$ reaches $2$, a bar in the barcode dies (corresponding to the top connected component) and another bar is born (the bottom component splits). 
\end{example}
\begin{figure}[!htb]
\centering
\centering
            \begin{subfigure}{0.3\textwidth}
                \centering
                \includegraphics[width=0.8\linewidth]{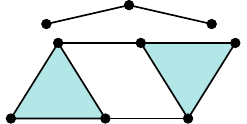}
                \caption{$\beta^{0}(X;\Z_2)=2$}
            \end{subfigure}
            \begin{subfigure}{0.3\textwidth}
                \centering
                \includegraphics[width=0.8\linewidth]{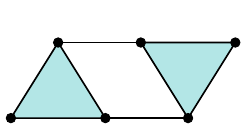}
                \caption{$\beta^{0,2}(X;\Z_2)=2$}
            \end{subfigure}
        \begin{subfigure}{0.35\textwidth}
            \centering
            \begin{tikzpicture}[scale=1]
                \draw (-1.8,0.75) node[right] {\tiny$H^{0}(X^\bullet;\Z_2)$};
                \draw[->] (-0.5,0) -- (3.5,0);
                \draw [-] (0,-0.05)--(0,1.35);
                \draw (3.5,0) node[right] {\small $q$};
                \draw(0,0)node[below]{\small $0$};
                \draw(1,0)node[below]{\small $1$};
                \draw(2,0)node[below]{\small $2$};
                \draw(3,0)node[below]{\small $3$};

                \draw [-] (1,-0.05)-- (1,0.05);
                \draw [-] (2,-0.05)-- (2,0.05);
                \draw [-] (3,-0.05)-- (3,0.05);

                \filldraw[fill=H0, draw=black] (0,0.3) rectangle (2,0.45);
                \filldraw[fill=H0, draw=black] (0,0.6) rectangle (3,0.75);
                \filldraw[fill=H0, draw=black] (2,0.9) rectangle (3,1.05);
            \end{tikzpicture}
            \vspace{-4.3ex}
            \caption{$H^0$ persistence barcode}
 \end{subfigure}
\caption{While the static thick Betti numbers in (a) and (b) indicate the same number of connected components, persistence distinguishes them, as reflected in the barcode of (c)}\label{Fig:RepresentacionCodigoBarras}
\end{figure}

\begin{theorem}\label{Prop:InvarianzaFunctoresCoesq}
Every simplicial isomorphism $X\simeq Y$ induces an isomorphism of persistence modules $H^n(X^\bullet;\Bbbk)\simeq H^n(Y^\bullet;\Bbbk)$ for any $n\geq 0$.
\end{theorem}
\begin{proof}
See Appendix~\ref{proof:Prop-InvarianzaFunctoresCoesq} for details.
\end{proof}

We now analyze the birth and death of thick holes, a phenomenon that is captured by  the decomposition of the persistence module and is visualized through barcodes. Following Theorem~\ref{Thm:Descomposicion}, the persistence module $H^n(X^\bullet;\Bbbk)$ decomposes into interval modules:
\begin{equation}
H^{n}(X^\bullet;\Bbbk)\simeq \bigoplus_{i=1}^m\Bbbk[b_i,d_i)\, \text{ with }\, b_i,d_i\in \mathbb{Z}_{\geq 0}\, \text{ and }\,  b_i<d_i\leq \dim\,X+1.
\end{equation}

\begin{definition}
Given a simplicial complex $X$, $n\geq 0$, and $0\leq b<d$, we define the \emph{\mbox{$(n,b,d)$-th} thick Betti number} of $X$ with coefficients in $\Bbbk$, and we denote it $\beta^n_{b,d}(X;\Bbbk)$, as the number of intervals of the persistence barcode associated to $H^n(X^\bullet;\Bbbk)$ of the form $[b,d)$. That is, if $H^n(X^\bullet;\Bbbk)\simeq \bigoplus_{i=1}^m \Bbbk[b_i,d_i)$, then 
\begin{equation*}
\beta^n_{b,d}(X;\Bbbk)= \#\{[b_i,d_i) : b_i=b\ \text{and}\ d_i=d\}_{1\leq i\leq m}.
\end{equation*}
\end{definition}

The $(n,b,d)$-th thick Betti number of $X$ counts the $n$-dimensional holes formed by removing simplices not contained in those of dimension greater $\geq b$, which are surrounded by $(d-1)$-simplices but not fully enclosed by $d$-simplices. As a straightforward consequence of 
Theorem~\ref{Prop:InvarianzaFunctoresCoesq} and the uniqueness of the interval decomposition in Theorem~\ref{Thm:Descomposicion}, we find that the $(n,b,d)$-th thick Betti numbers are simplicial invariants.
\begin{corollary}
If $X$ and $Y$ are isomorphic simplicial complexes, then $\beta^n_{b,d}(X;\Bbbk)=\beta^n_{b,d}(Y;\Bbbk)$ for all $n\geq 0$ and $0\leq b<d$.
\end{corollary}

The following proposition provides information on the potential birth and death values in the barcode of $H^n(X^\bullet;\Bbbk)$: it shows that equivalence classes in $H^n(X^\bullet;\Bbbk)$ are born either at $q=0$ or at $q\geq n+2$ and that if $H^n(X;\Bbbk)$ has a non-trivial equivalence class, it cannot die until $q=n+1$. This corresponds with the intuition that we cannot create new $n$-dimensional holes until we remove $(n+1)$-dimensional simplices and that we cannot break $n$-dimensional holes until we remove the $n$-dimensional simplices.

\begin{proposition}\label{Prop:aibiValues}
If $q\leq n$, the linear map $H^n(X;\Bbbk)\to H^n(X^q;\Bbbk)$ induced by the inclusion $X^q\subseteq X$ is the identity map. Moreover, the induced morphism $H^n(X;\Bbbk)\to H^n(X^{n+1};\Bbbk)$ is surjective.
\end{proposition}
\begin{proof}
See Appendix~\ref{proof:Prop-aibiValues} for details.
\end{proof}

\begin{corollary}
    The possible non-trivial $(n,b,d)$-th thick Betti numbers are those for which $b=0$ and $d\geq n+1$ or $d>b\geq n+2$.
\end{corollary}

 The previous corollary distinguishes two different types of features in the persistence barcode. The second case (where $d > b \geq n+2$) corresponds to emergent holes: features that do not exist in the original complex but appear as artifacts of the cofiltration process. Conversely, to assess the robustness of the original simplicial complex, we focus on the first case ($b=0$). For these features, the death time $d$ identifies the weakest link of the cycle: a hole dying at the $d$-th coskeleton is supported by at least one facet of dimension exactly $d-1$. Thus, $\beta^n_{0,d}(X;\Bbbk)$ counts the $n$-dimensional holes in $X$ whose defining facets have a minimum dimension of $d-1$. These death values act as a natural measure of thickness. In particular, as the following corollary states, this yields a stratification of the holes in $X$ relative to the dimension of their defining simplices (Figure~\ref{Fig:nbdCoskeletal}).

\begin{corollary}\label{Cor:StratificationCoskeletal}
Let $X$ be a simplicial complex. Then  $\beta^n_{0,d}(X;\Bbbk)=0$ for all $d\leq n$ and 
\begin{equation}
\beta^n(X;\Bbbk)=\sum_{d=n+1}^{\dim X+1} \beta^n_{0,d}(X;\Bbbk)\,.
\end{equation}
\end{corollary}

\begin{example}
Consider the simplicial complex $X$ in Figure~\ref{Fig:nbdCoskeletal}a, which has two holes ($\beta^1(X;\Z_2)=2$). Figure~\ref{Fig:nbdCoskeletal}b shows, in orange and dashed purple, two representatives of independent cohomology classes in $H^1(X;\Z_2)$. Only the orange one survives the elimination of the middle edge when taking the $2$-coskeleton. This is encoded in the persistence barcode (Figure~\ref{Fig:nbdCoskeletal}d), from which we obtain $\beta^1_{0,2}(X;\Z_2)=1$ (the minimum dimension of the facets defining the dashed cycle is $1$) and $\beta^1_{0,3}(X;\Z_2)=1$ (the orange hole is surrounded by triangles). However, it is important to note that this stratification cannot be achieved by analyzing the dimension of the facets containing the representatives of cohomology classes defining a basis of $H^n(X;\Bbbk)$. For example, the orange and dashed purple cycles in figure (c) also form a basis of $H^1(X;\Z_2)$ but both include an edge which is not contained in a triangle. 
\end{example}
    \begin{figure}[htb!]
\centering
\begin{subfigure}{0.15\textwidth}
\centering
    \includegraphics[width=0.9\textwidth]{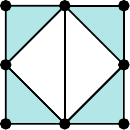}
\caption{}
\end{subfigure}
\hspace{2ex}
\begin{subfigure}{0.15\textwidth}
\centering
    \includegraphics[width=0.9\textwidth]{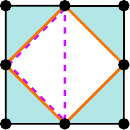}
\caption{}
\end{subfigure}
\hspace{2ex}
\begin{subfigure}{0.15\textwidth}
\centering
    \includegraphics[width=0.9\textwidth]{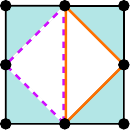}
\caption{}
\end{subfigure}
\begin{subfigure}{0.32\textwidth}
\centering
\includegraphics[width=1\textwidth]{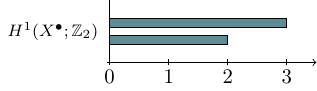}
            \caption{}
           \end{subfigure}
\caption{Simplicial complex (a), representatives of a basis of $H^1(X;\Z_2)$ in orange and dashed purple (b and c) and thickness persistence barcode (d)}\label{Fig:nbdCoskeletal}
\end{figure}

\section{How cohesive is a thick hole?}\label{Sec:Stratified}

As discussed in the introduction, analyzing the robustness of a simplicial complex involves two steps: measuring cycle thickness (using the thick Betti numbers introduced in the previous section) and evaluating the strength of their internal connections, measured by the higher-order adjacency between them. Here, we focus on the second aspect.

\subsection{Cohesive Betti numbers}

Our goal is to analyze the topological impact of removing simplices of fixed dimensions from a simplicial complex. Notice that the resulting structure is not necessarily a simplicial complex but rather a finite poset. Finite posets thus provide a suitable general framework for analyzing such cases.

In this section, we use cohomology of finite posets to introduce a new set of invariants that give us information about the cohesion of cohomological cycles in a simplicial complex. This phenomenon is illustrated in Figure~\ref{Fig:EliminacionVertices}: eliminating the vertices separates the triangles on the left while leaving the ones on the right still connected.

\begin{figure}[htb!]
\centering
\includegraphics[scale=0.8]{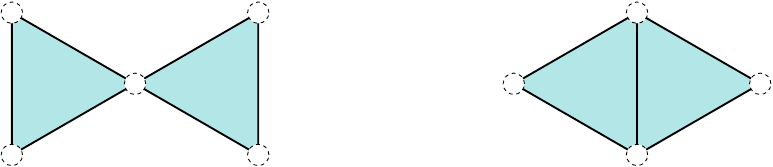}
\caption{Removing vertices from a simplicial complex}\label{Fig:EliminacionVertices}
\end{figure}

The dimensions of the simplices of a simplicial complex $X$ define a stratification of the points in the face poset $P_X$, as shown in its Hasse diagram (Figure~\ref{Fig:HasseDiagram}a). This natural partitioning motivates the study of subposets formed by selecting only certain dimensional layers.

\begin{definition}
Let $\mathbf{h}=\{h_0,h_1,\dots,h_m\}$ be a set of integers such that $0\leq h_0<h_1<\dots<h_m$. We define the \emph{$\mathbf{h}$-face poset} of $X$, and denote it by $P_X^\mathbf{h}$, as the subposet defined by the strata of $P_X$ corresponding to the dimensions $\{h_0,h_1,\dots, h_m\}$. That is,
\begin{equation*}
P_X^{\mathbf{h}}= \{ \sigma\in P_X : \dim \,\sigma= h_i \,   \text{ for some }0\leq i\leq m\}.
\end{equation*} 
\end{definition}

\begin{example}
Consider the simplicial complex in Figure~\ref{Fig:ExCosk23Triang}a. In Figure~\ref{Fig:HasseDiagram}a, we present the Hasse diagram of its associated face poset, where the triangles, edges, and vertices from Figure~\ref{Fig:ExCosk23Triang}a correspond to $2$-dimensional, $1$-dimensional, and $0$-dimensional vertices, respectively, in the Hasse diagram. Figure~\ref{Fig:HasseDiagram}b represents its associated $\{0,2\}$-face poset.\end{example}

\begin{figure}[!htb]
\centering
\begin{minipage}{0.44\textwidth}
\centering
\includegraphics[scale=0.75]{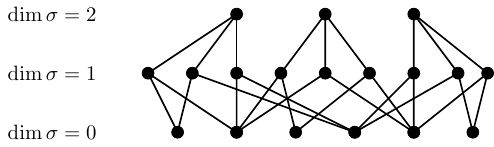}
\subcaption{$P_X$}
\end{minipage}
\begin{minipage}{0.44\textwidth}
\centering
\includegraphics[scale=0.75]{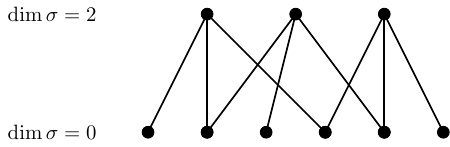}
\subcaption{$P_X^{\{0,2\}}$}
\end{minipage}
\caption{Hasse diagrams of the face poset (a) of the simplicial complex in Figure~\ref{Fig:ExCosk23Triang}a and its $\{0,2\}$-face poset (b)}\label{Fig:HasseDiagram}
\end{figure} 

\begin{definition}\label{d:cohesivebetti}
The {\em$(n,\mathbf{h})$-th cohesive Betti number} of a simplicial complex $X$ with coefficients in $\Bbbk$ is defined as the dimension of the $n$-th cohomology space with coefficients in $\Bbbk$ of the $\mathbf{h}$-face poset, 
\begin{equation*}
\beta^{n,\mathbf{h}}(X;\Bbbk)=\dim\,H^n(P_X^\mathbf{h};\Bbbk)\,.
\end{equation*}
\end{definition}

Taking $\mathbf{h}=[\dim X]=\{0,1,\dots,\dim X\}$ one recovers the usual Betti numbers of the simplicial complex, i.e., $\beta^{n,[\dim X]}(X;\Bbbk)=\beta^n(X;\Bbbk)$ (by Corollary~\ref{Cor:BettiNumbersSheafCohomology}). Thus, like the thick Betti numbers, the set of cohesive Betti numbers of $X$ extends the set of ordinary Betti numbers. 

\begin{remark}
Notice that $(\beta^{0,\{0\}}(X;\Bbbk),\, \beta^{0,\{1\}}(X;\Bbbk),\dots,\,\beta^{0,\{\dim X\}}(X;\Bbbk))$ is the so-called \mbox{$f$-vector} of $X$ and $(\beta^{0,[\dim X]}(X;\Bbbk),\, \beta^{1,[\dim X]}(X;\Bbbk),\dots,\, \beta^{\dim X,[\dim X]}(X;\Bbbk))$ is the \mbox{$\beta$-vector} of $X$.
\end{remark}

Let us consider a simplicial isomorphism $f\colon X\stackrel{\sim}\to Y$. Setting $P_f(\sigma)\coloneqq f(\sigma)$ yields an isomorphism between the face posets $P_f\colon P_X\to P_Y$, whose restriction to $P_X^\mathbf{h}$ gives an isomorphism between $P_X^\mathbf{h}$ and $P_Y^\mathbf{h}$. Consequently, by the functoriality of poset cohomology, one concludes that  $H^n(P_X^\mathbf{h};\Bbbk)\simeq H^n(P_Y^\mathbf{h};\Bbbk)$ for all $n\geq 0$. Hence, we obtain the following result regarding the simplicial invariance for cohesive Betti numbers.
\begin{theorem}\label{t:cohesivebetti} 
If $X$ and $Y$ are isomorphic simplicial complexes, then $\beta^{n,\mathbf{h}}(X;\Bbbk)=\beta^{n,\mathbf{h}}(Y;\Bbbk)$ for all $n\geq 0$ and all $\mathbf{h}=\{h_0,h_1,\dots, h_m\}$ such that $0\leq h_0<h_1<\dots<h_m$. 
\end{theorem}

    The order complex of the $\mathbf{h}$-face poset of $X$ is a simplicial subcomplex of the barycentric subdivision of $X$:
\begin{equation}\label{e:BS}
        \mathcal{K}(P_X^{\mathbf{h}})= \{(\sigma_0,\sigma_1,\dots,\sigma_n)\in \mathcal{K}(X) : \dim \sigma_i\in \mathbf{h} \text{ for all } 0\leq i\leq n\}.
\end{equation}
    That is, $\mathcal{K}(P_X^{\mathbf{h}})$ is the simplicial complex obtained by removing from $\mathcal{K}(X)$ the stars of the vertices of $\mathcal{K}(X)$ associated with the simplices in $X$ whose dimension is not in $\mathbf{h}$:
   \begin{equation*}
\mathcal{K}(P_X^\mathbf{h})=\mathcal{K}(X)\smallsetminus\bigcup_{\substack{\dim\sigma\notin \mathbf{h}\\[0.8ex] \sigma \in X}}\mathrm{st}_{\mathcal{K}(X)}(\sigma),
\end{equation*}
where the star of a vertex $\sigma\in \mathcal{K}(X)$ is defined as $\mathrm{st}_{\mathcal{K}(X)}(\sigma) \coloneqq \{\tau \in \mathcal{K}(X) : \sigma \in \tau\}$. Note that deleting the stars of a set of vertices from a simplicial complex always leaves a valid simplicial complex, since the property of being closed under taking subsets is preserved.

\begin{example}
Let $X$ be the simplicial complex in Figure~\ref{Fig:ExCosk23Triang}a. Taking $\mathbf{h}=\{0,2\}$, that is, removing the edges from $X$, we obtain the following cohesive Betti numbers
\begin{equation*}
    \beta^{0,\{0,2\}}(X;\Z_2)=1 \ \ \text{ and }\ \ \beta^{1,\{0,2\}}(X;\Z_2)=1.
\end{equation*}
They correspond with the topological properties of the order complexes of the $\{0,2\}$-face poset of $X$ (Figure~\ref{Fig:ExCoskBaricentrica}a): the simplicial complex $\mathcal{K}(P_X^{\{0,2\}})$ has a connected component and a hole. 

When we instead remove the vertices, $\mathbf{h}=\{1,2\}$, we obtain cohesive Betti numbers
\begin{equation*}
    \beta^{0,\{1,2\}}(X;\Z_2)=3 \ \ \text{ and }\ \ \beta^{1,\{1,2\}}(X;\Z_2)=0.
\end{equation*}
This corresponds to the three connected components and the absence of holes in $\mathcal{K}(P_X^{\{1,2\}})$ (Figure~\ref{Fig:ExCoskBaricentrica}b).
\end{example}
    \begin{figure}[!htb]
\centering
\begin{subfigure}{0.45\textwidth}
\centering
\includegraphics[scale=1]{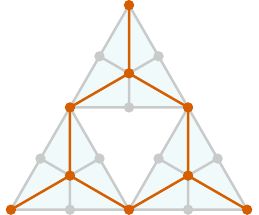}
\caption{$\mathcal{K}(P_X^{\{0,2\}})$}
\end{subfigure}
\begin{subfigure}{0.45\textwidth}
\centering
\includegraphics[scale=1]{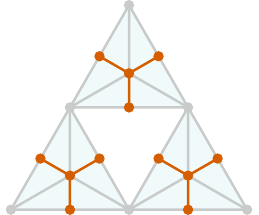}
\caption{
  $\mathcal{K}(P_X^{\{1,2\}})$
}
\end{subfigure}
\caption{Barycentric subdivision of simplicial complex in Figure~\ref{Fig:ExCosk23Triang}a and the corresponding subcomplexes $\mathcal{K}(P_X^{\{0,2\}})$ and $\mathcal{K}(P_X^{\{1,2\}})$ (marked in orange)}\label{Fig:ExCoskBaricentrica}
\end{figure}

We now restrict our attention to the $n=0$ cohesive Betti numbers. To understand the information given by this set of invariants, let us consider, for each $0\leq h_0<h_1$ and each simplicial complex $X$, the following vector space:
\begin{equation*}
\Gamma^{h_0,h_1}(X;\Bbbk)\coloneqq \{x\in \prod_{\substack{\sigma\in X\\ \dim \,\sigma=h_0}}\Bbbk\ : \ x_{\sigma}=x_{\sigma'} \ \, \forall\,\sigma,\sigma' \face \tau\in X : \dim\tau\geq h_1\}.
\end{equation*}

\begin{proposition}\label{Prop:StratifiedCaminos}
The $\Bbbk$-vector spaces $\Gamma^{h_0,h_1}(X;\Bbbk)$ and $H^0(P_X^\mathbf{h};\Bbbk)$ are isomorphic for any \mbox{$\mathbf{h}=\{h_0,h_1,\dots,h_m\}$.}
\end{proposition}
\begin{proof}
See Appendix~\ref{proof:Prop-StratifiedCaminos} for details.
\end{proof}

\begin{corollary}\label{cor:StratifiedCaminos}
The $(0,\mathbf{h})$-th cohesive Betti number of $X$ counts the number of $(h_0,h_1)$-connected components of $X$,  
\begin{equation*}
    \beta^{0,\mathbf{h}}(X;\Bbbk)=\#\faktor{S^{h_0}(X)}{\sim_{h_0h_1}}.
\end{equation*}
\end{corollary}
\begin{proof}
See Appendix~\ref{proof:cor-StratifiedCaminos} for details.
\end{proof}

Therefore, just as the $0$-th Betti number provides information about the connectivity of the simplicial complex, the cohesive Betti numbers associated with the $0$-th cohomology spaces give information about the connectivity of simplices via higher-order adjacencies. Proposition~\ref{Prop:Coskeletal0h} and Corollary \ref{cor:StratifiedCaminos} imply the following formula.

\begin{corollary}
For $\mathbf{h}=\{0,h_1,\dots,h_m\}$, the \mbox{$(0,\mathbf{h})$-th} cohesive Betti number and the $(0,h_1)$-th thick Betti number of $X$ are related as follows:
\begin{equation*}
    \beta^{0,\{0,h_1,\dots,h_m\}}(X;\Bbbk)=\beta^{0,h_1}(X;\Bbbk)+\#\{v\in V(X) : v\centernot{\lhd}\sigma \ ,\, \forall \sigma\in S^{h_1}(X)\}.
\end{equation*}
\end{corollary}

\begin{remark}
Although the classification given by the cohesive Betti numbers is generally finer than that given by the thick Betti numbers (mainly because the former encode the number of simplices in each dimension by taking $\mathbf{h}=\{h\}$), there are simplicial complexes with the same cohesive Betti numbers but with different thick Betti numbers (Figure~\ref{Fig:ComparacionCoskStrati} and Table \ref{Tab:ComparacionCoskStrat}).
\end{remark}
\begin{figure}[htb!]
    \centering
    \includegraphics[width=0.5\linewidth]{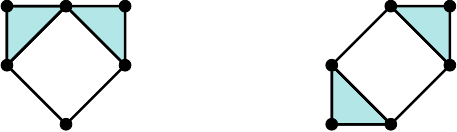}
    \caption{Two simplicial complexes with the same cohesive Betti numbers but different thick Betti numbers (Table \ref{Tab:ComparacionCoskStrat})}\label{Fig:ComparacionCoskStrati}
\end{figure}

\begin{table}[htb!]
\centering
\begin{tabular}{c | cc || c | cc}
\toprule
\multicolumn{3}{c||}{Cohesive Betti numbers} & \multicolumn{3}{c}{Thick Betti numbers} \\ \midrule
\multicolumn{1}{l}{}  &  Left & Right & \multicolumn{1}{l}{}  &  Left & Right \\ 
\midrule
        $\beta^{0,\{0\}}(X;\Z_2)$      &  $6$ &  $6$ & $\beta^{0,0}(X;\Z_2)$ & $1$ & $1$\\
        $\beta^{0,\{1\}}(X;\Z_2)$      &  $8$ &  $8$ &      $\beta^{1,0}(X;\Z_2)$ & $1$ & $1$ \\
        $\beta^{0,\{2\}}(X;\Z_2)$      &  $2$ &  $2$ & $\beta^{0,1}(X;\Z_2)$ & $1$ & $1$ \\
        $\beta^{0,\{0,1\}}(X;\Z_2)$    &  $1$ &  $1$ & $\beta^{1,1}(X;\Z_2)$ & $1$ & $1$\\
        $\beta^{1,\{0,1\}}(X;\Z_2)$    &  $3$ &  $3$ & $\pmb{\beta^{0,2}(X;\Z_2)}$ & $\pmb{1}$ & $\pmb{2}$ \\
        $\beta^{0,\{0,2\}}(X;\Z_2)$    &  $2$ &  $2$ & $\beta^{1,2}(X;\Z_2)$ & $0$ & $0$ \\
        $\beta^{1,\{0,2\}}(X;\Z_2)$    &  $0$ &  $0$ \\
        $\beta^{0,\{1,2\}}(X;\Z_2)$    &  $4$ &  $4$ \\
        $\beta^{1,\{1,2\}}(X;\Z_2)$    &  $0$ &  $0$ \\
        $\beta^{0,\{1,2,3\}}(X;\Z_2)$  &  $1$ &  $1$ \\
        $\beta^{1,\{1,2,3\}}(X;\Z_2)$  &  $1$ &  $1$ 
\end{tabular}
\caption{Betti numbers of the simplicial complexes in Figure~\ref{Fig:ComparacionCoskStrati}}\label{Tab:ComparacionCoskStrat}
\end{table}

\subsection{Persistence of a hole's cohesion}\label{ss:cohesivepersistence}

Similar to thick Betti numbers, the cohesion of homology cycles as a robustness indicator in a simplicial complex should be considered within the context of persistence, which can be studied by means of the linear map
\begin{equation}\label{eq:holesmap}
\varphi^{n,\mathbf{h}}\colon H^n(X;\Bbbk)\longrightarrow H^n(P_X^\mathbf{h};\Bbbk)
\end{equation}
induced by the inclusion $P_X^\mathbf{h}\subseteq P_X$ and the isomorphism $H^n(X;\Bbbk)\simeq H^n(P_X;\Bbbk)$.

One can interpret the kernel, image, and cokernel of this linear map in the following way:

\begin{itemize}
\item The kernel of $\varphi^{n,\mathbf{h}}$ corresponds to the holes in $X$ that disappear when removing the simplices whose dimension is not in $\mathbf{h}$. 
\item The image of $\varphi^{n,\mathbf{h}}$ counts the holes of $X$ that persist after the elimination of simplices.
\item The cokernel of $\varphi^{n,\mathbf{h}}$ captures the new holes that arise in the construction.
\end{itemize}

\begin{remark}
As with thick Betti numbers, this process of simplicial removal can lead to the emergence of new holes that were not present in the original simplicial complex, as shown in Figure~\ref{Fig:CreacionAgujeroStratified}.
\end{remark}

\begin{figure}[!htb]
\centering
\begin{subfigure}{0.45\textwidth}
\centering
\includegraphics[width=0.6\linewidth]{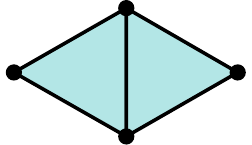}
\caption{$\beta^{1}(X;\Z_2)=0$}
\end{subfigure}
\begin{subfigure}{0.45\textwidth}
\centering
\includegraphics[width=0.6\linewidth]{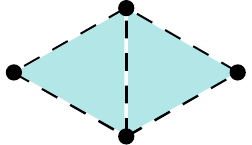}
\caption{$\beta^{1,\{0,2\}}(X;\Z_2)=1$}
\end{subfigure}
\caption{Removing all edges from the simplicial complex creates a 1-dimensional hole (specifically due to the removal of the intersecting edge), as captured by the cohesive Betti numbers}\label{Fig:CreacionAgujeroStratified}
\end{figure}

\begin{definition}
We define the persistent $(n,\mathbf{h})$-cohesive Betti number by:
\begin{equation*}
    \beta^{n,\mathbf{h}}_{\img }(X;\Bbbk)=\dim \img\varphi^{n,\mathbf{h}}.
\end{equation*}
\end{definition}

Notice that it holds that $\beta^{n,\mathbf{h}}_{\img}(X;\Bbbk)\leq \beta^{n}(X;\Bbbk)$. By comparing $\beta^{n}(X;\Bbbk)$ and $\beta^{n,\mathbf{h}}_{\img}(X;\Bbbk)$, we can measure the robustness of holes in terms of the strength of their internal connections or adjacencies, meaning their level of cohesion. The closer $\beta^{n,\mathbf{h}}(X;\Bbbk)$ is to $\beta^n(X;\Bbbk)$ for each $\mathbf{h}$, the more cohesive the $n$-dimensional holes are. 
 
 \begin{example}
Observe Figure~\ref{Fig:TresComplejosStratified}. The hole in the simplicial complex on the left is defined by a closed $(1,2)$-walk, meaning it persists even when removing the edges or vertices of $X$. In the second simplicial complex, the cycle is determined by a closed $(0,2)$-walk. Removing the vertices destroys the hole, as there is no closed walk of $1$-lower adjacent triangles surrounding it, which is captured by $\beta^{1,\{1,2\}}_{\img}(Y;\Z_2)$. In the last simplicial complex, the hole is not defined by triangles and thus does not persist in either case.
\end{example}

\begin{figure}[htb!]
\centering
\includegraphics[width=0.9\linewidth]{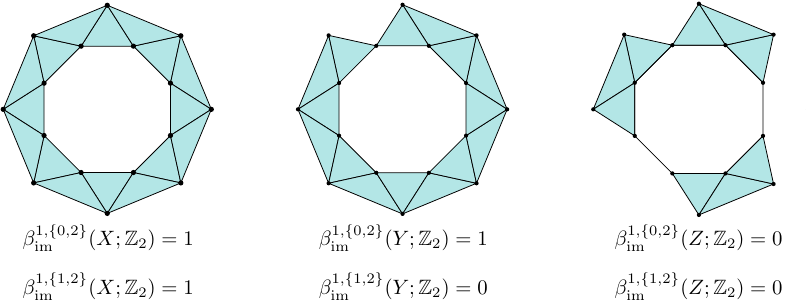}
\caption{Cohesive Betti numbers distinguish holes in simplicial complexes based on the dimension and adjacency of the surrounding simplices}\label{Fig:TresComplejosStratified}
\end{figure}

\section{Resilience: thick and cohesive bipersistence}\label{Sec:Resilience}

This section studies resilience as a persistence property of connected components and holes under prescribed simplicial degradation processes. Given a cofiltration modeling an attack, obtained by progressively removing vertices, edges, triangles, or higher-dimensional simplices, we combine this parameter with a second structural parameter measuring either thickness or cohesiveness. This produces a biparameter persistence module. In order to extract interpretable information, we fix one structural parameter and study the resulting ladder modules, together with the one-parameter persistence modules given by their image, kernel, and cokernel. These modules allow us to distinguish features that remain thick or cohesive during the degradation process, features that lose this structural support before disappearing, and features that are created by the removal of simplices.

\subsection{Resilience of thick holes}

Let
\begin{equation}\label{Eq:Cofiltration}
   X^{\bullet}\colon X^0\supseteq X^1\supseteq X^2\supseteq \dots \supseteq X^M
\end{equation}
be a cofiltration of simplicial complexes modeling an arbitrary degenerative process on an initial simplicial complex $X^0$ of dimension $N$ (in which vertices, edges, triangles, etc., may disappear at each step). As illustrated in Sect.~\ref{Subsec:ThicknessPersistence}, a comprehensive study of the thickness of the holes requires a dynamic perspective. In this regard, for each $X^i$, we consider the cofiltration defined by its coskeletons: 
\begin{equation*}
    X^{i,\bullet}\colon X^{i,0}\supseteq X^{i,1}\supseteq X^{i,2}\supseteq \dots \supseteq X^{i,N},
\end{equation*}
where $X^{i,q}$ denotes the $q$-coskeleton of $X^i$. Varying both parameters, $i=0,1,\dots, M$ and $q=0,1,\dots, N$, we obtain a biparameter persistence module 
\begin{equation}\label{Eq:PersistenceBimoduleThick}
    \begin{tikzcd}
    H^n(X^{0,N};\Bbbk) \arrow[r] & H^n(X^{1,N};\Bbbk) \arrow[r] & H^n(X^{2,N};\Bbbk) \arrow[r] & \dots \arrow[r] & H^n(X^{M,N};\Bbbk) \\ 
    \vdots \arrow[u] \arrow[r, phantom, ""{coordinate, name=Z}] & \vdots \arrow[u] \arrow[r, phantom, ""{coordinate, name=Z}] & \vdots \arrow[u] \arrow[r, phantom, ""{coordinate, name=Z}] & \dots \arrow[r, phantom, ""{coordinate, name=Z}] & \vdots \arrow[u] \\[-3ex]
 H^n(X^{0,1};\Bbbk)  \arrow[r] & H^n(X^{1,1};\Bbbk)  \arrow[r] & H^n(X^{2,1};\Bbbk)  \arrow[r] & \dots \arrow[r] & H^n(X^{M,1};\Bbbk)  \\
 H^n(X^0;\Bbbk) \arrow[u] \arrow[r] & H^n(X^1;\Bbbk) \arrow[u] \arrow[r] & H^n(X^2;\Bbbk) \arrow[u] \arrow[r] & \dots \arrow[r] & H^n(X^M;\Bbbk) \arrow[u]
   \end{tikzcd}
\end{equation}
The horizontal axis ($i$) reveals how features emerge or disappear as simplices are removed from the initial complex, while the vertical axis ($q$) tracks how features change when restricting to higher-dimensional simplices. 

As we mentioned at the end of Section~\ref{Subsec:Persistence}, a limitation of multipersistence is the absence, in general, of a discrete invariant that classifies persistence modules in the same way that the barcode classifies one-parameter persistence modules. For this reason, and inspired by the framework of~\cite{CohenSteiner_Otros09}, we extract one-parameter information by fixing a value of the structural parameter. More precisely, for each $q=1,\dots,N$ we fix the thickness parameter and study the resulting ladder module
\begin{equation} \label{e:Thick6pack}
    \begin{tikzcd}
 H^n(X^{0,q};\Bbbk)  \arrow[r] & H^n(X^{1,q};\Bbbk)  \arrow[r] & H^n(X^{2,q};\Bbbk)  \arrow[r] & \dots \arrow[r] & H^n(X^{M,q};\Bbbk)  \\
 H^n(X^0;\Bbbk) \arrow[u, "\varphi^{n,0,q}"] \arrow[r] & H^n(X^1;\Bbbk) \arrow[u, "\varphi^{n,1,q}"] \arrow[r] & H^n(X^2;\Bbbk) \arrow[u,"\varphi^{n,2,q}"] \arrow[r] & \dots \arrow[r] & H^n(X^M;\Bbbk) \arrow[u,"\varphi^{n,M,q}"]
    \end{tikzcd}
\end{equation}
which allows us to analyze the subsequent collection of one-parameter persistence modules:

\begin{itemize}
    \item The $n$-th persistence module of the cofiltration modeling the degenerative process, on the bottom horizontal line,
\begin{equation*}
    H^n(X^\bullet; \Bbbk) : H^n(X^0; \Bbbk) \to H^n(X^{1}; \Bbbk) \to \cdots \to H^n(X^{M}; \Bbbk) \to 0.
\end{equation*}
    \item The $n$-th persistence module of the $q$-coskeleton of the cofiltration, on the top horizontal line, 
    \begin{equation*}
        H^n(X^{\bullet,q}; \Bbbk) : H^n(X^{0,q}; \Bbbk) \to H^n(X^{1,q}; \Bbbk) \to \cdots \to H^n(X^{M,q}; \Bbbk) \to 0.
    \end{equation*}
    \item The image persistence module
    \begin{equation*}
        \img \varphi^{n,\bullet,q} : \img \varphi^{n,0,q} \to \img \varphi^{n,1,q} \to \cdots \to \img \varphi^{n,M,q} \to 0.
    \end{equation*}
    It registers the persistence of holes enclosed by simplices of dimension at least $q$ within the cofiltration $X^\bullet$.
    \item The kernel persistence module
    \begin{equation*}
        \ker \varphi^{n,\bullet,q} : \ker \varphi^{n,0,q} \to \ker \varphi^{n,1,q} \to \cdots \to \ker \varphi^{n,M,q} \to 0.
    \end{equation*}
    It captures the evolution of holes that are enclosed by simplices of dimension less than $q$ as the cofiltration $X^\bullet$ progresses.
    \item The cokernel persistence module
    \begin{equation*}
        \coker \varphi^{n,\bullet,q} : \coker \varphi^{n,0,q} \to \coker \varphi^{n,1,q} \to \cdots \to \coker \varphi^{n,M,q} \to 0.
    \end{equation*}
 It encodes the persistence of cohomology classes that emerge when simplices outside the $q$-coskeleton are removed.
\end{itemize}
Let us illustrate this study with the following examples.
\begin{example}\label{Ex:PersistentThick}
 In Figure~\ref{Fig:FiltrationPersistentCosk}a we have a cofiltration of simplicial complexes with $5$ steps. In the first step, $X^0$, there is a $1$-dimensional hole that is fully enclosed by triangles ($2$-simplices). Along the cofiltration, some of the triangles surrounding the hole disappear, resulting in a thinning of the hole. However, the persistence barcode of the cofiltration does not capture any change in the structure of the simplicial complexes (Figure~\ref{Fig:FiltrationPersistentCosk}c). Examining the cofiltration of $2$-coskeletons (Figure~\ref{Fig:FiltrationPersistentCosk}b), we observe that the hole disappears at step 4. The persistence barcode of the cofiltration $X^{\bullet,2}$ (Figure~\ref{Fig:FiltrationPersistentCosk}d) reflects that the hole in $X^\bullet$ is no longer enclosed by triangles in $X^4$. In this way, although the hole persists throughout the cofiltration, we can see that it is suffering some damage, causing a loss of thickness that could indicate a future disappearance of the hole. It is worth noting that, in this example, the image, kernel, and cokernel persistence modules offer little additional insight beyond what is already captured by the persistence module of the $2$-coskeleton. In Example~\ref{ex:kercokerbarcode} we present a case where such information is less immediately accessible.
\end{example}
\begin{figure}[!htb]
    \centering
    \begin{subfigure}{0.9\textwidth}
         \includegraphics[scale=0.5]{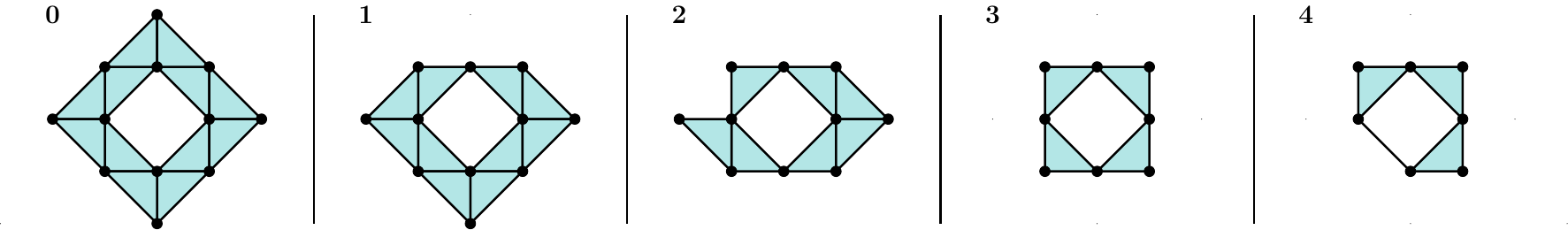}
         \caption{Cofiltration $X^\bullet$}\label{Fig:FiltrationPersistentCoskA}
    \end{subfigure}
    
 \vspace{3ex}
 
 \begin{subfigure}{0.9\textwidth}
         \includegraphics[scale=0.5]{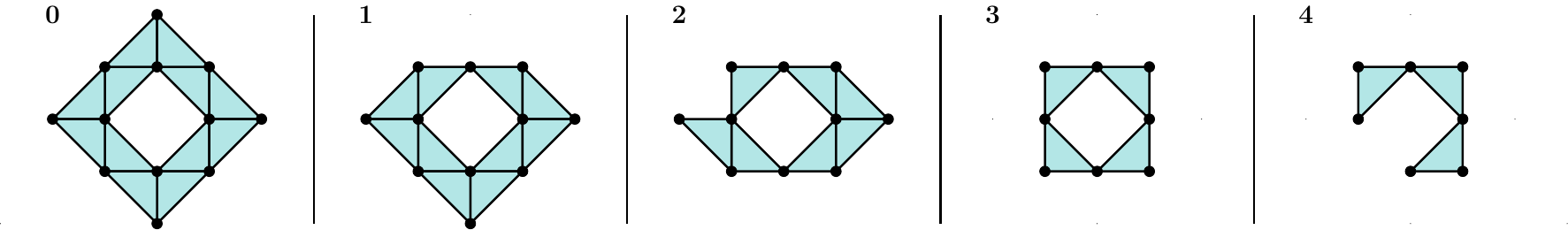}
         \caption{Cofiltration $X^{\bullet,2}$}\label{Fig:FiltrationPersistentCoskB}
    \end{subfigure}

 \vspace{3ex}
   
    \begin{minipage}{0.49\textwidth}
        \centering
        \includegraphics[scale=0.8]{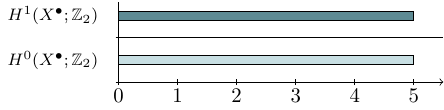}
        \subcaption{Barcodes of cofiltration $X^\bullet$} \label{Fig:FiltrationPersistentCoskC}
         \vspace{2ex}
    \end{minipage}
    \begin{minipage}{0.49\textwidth}
        \centering
        \includegraphics[scale=0.8]{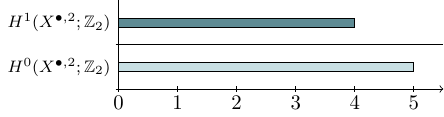}
        \subcaption{Barcodes of cofiltration $X^{\bullet,2}$}\label{Fig:FiltrationPersistentCoskD}
         \vspace{2ex}
    \end{minipage}
    \caption{Cofiltration of simplicial complexes, cofiltration of $2$-coskeletons and their corresponding barcodes}\label{Fig:FiltrationPersistentCosk}
\end{figure}

\begin{example}\label{ex:kercokerbarcode}
Figure~\ref{Fig:Filtration6PackThick} shows a cofiltration and the associated persistence barcodes for $q=2$. We focus here on $0$-th cohomology, that is, on the evolution of connected components. In the cofiltration $X^\bullet$ there are two connected components: one persists throughout the entire cofiltration, and another is born at step 1 and dies at step 3.

The image barcode indicates that the long-persistent connected component in $X^\bullet$ is defined by triangles whereas the shorter-lived component is defined by lower dimensional simplices at step $2$, which is when its corresponding bar is born in the kernel persistence module. The bar in $\ker \varphi^{0,\bullet,2}$ reflects the thinning of the connected component, anticipating its disappearance. Finally, the bar in the cokernel persistence module reflects the connected component that emerges at step 0 when passing to the $2$-coskeleta, which becomes a connected component of $X^\bullet$ at step 1.
\end{example}

\begin{figure}[!htb]
    \centering
\begin{subfigure}{0.9\textwidth}
    \centering
    \includegraphics[width=1\linewidth]{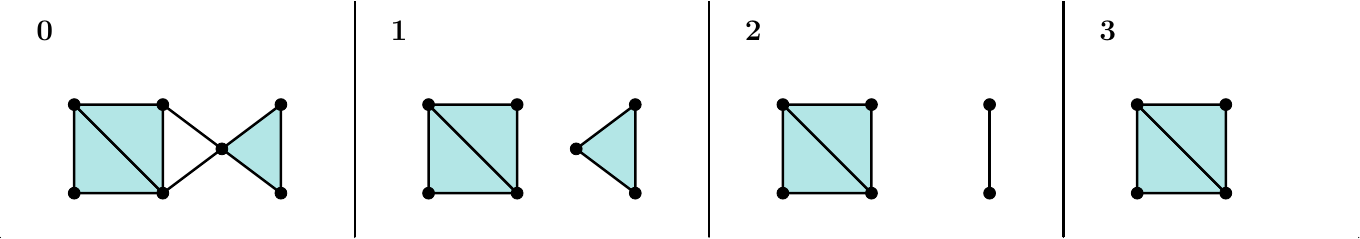}
    \subcaption{Cofiltration $X^\bullet$}
\end{subfigure}
\par\bigskip
\begin{subfigure}{0.48\textwidth}
    \centering
    \includegraphics[width=0.9\linewidth]{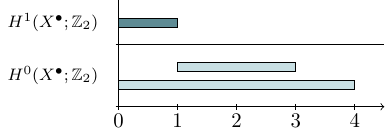}
    \subcaption{\footnotesize{Barcodes of the cofiltration $X^\bullet$}}
\end{subfigure}
  \hfill
\begin{subfigure}{0.48\textwidth}
    \centering
    \includegraphics[width=0.9\linewidth]{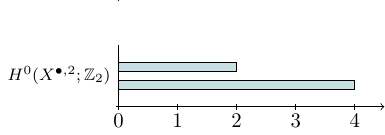}
    \subcaption{\footnotesize{Barcode of the cofiltration of $2$-coskeletons}}
\end{subfigure} 
\par\bigskip
\begin{subfigure}{0.32\textwidth}
    \centering
    \includegraphics[width=0.9\linewidth]{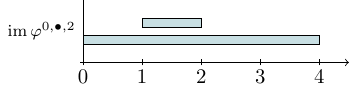}
    \subcaption{\footnotesize{Barcode of the image persistence module}}
\end{subfigure}
  \hfill
\begin{subfigure}{0.32\textwidth}
    \centering
    \includegraphics[width=0.9\linewidth]{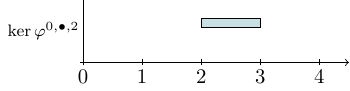}
    \subcaption{\footnotesize{Barcode of the kernel persistence module}}
\end{subfigure}
  \hfill
\begin{subfigure}{0.32\textwidth}
    \centering
    \includegraphics[width=0.9\linewidth]{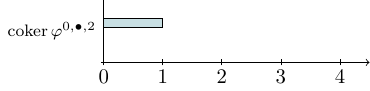}
    \subcaption{\footnotesize{Barcode of the cokernel persistence module}}
\end{subfigure}
\caption{Cofiltration of simplicial complexes and persistence barcodes arising from the $2$-coskeletons}\label{Fig:Filtration6PackThick}
\end{figure}

\subsection{Resilience of cohesive holes}

Example~\ref{Ex:PersistentThick} illustrates the utility of thick Betti numbers in analyzing the degenerative dynamics of a simplicial complex. However, Figure~\ref{Fig:FiltrationPersistentCosk} also shows that thickness alone does not fully characterize the degradation process. Specifically, it fails to capture that the hole is no longer supported by a closed $(1,2)$-walk at step $1$, as well as the progressive structural weakening in steps $2$ and $3$. This limitation, which arises from the inability to track higher-order adjacencies among the simplices supporting cohomological cycles, is precisely the gap that the resilience analysis based on cohesive Betti numbers is designed to address.

Given the cofiltration in equation~\eqref{Eq:Cofiltration}: 
\begin{equation*}
     X^{\bullet}\colon X^0\supseteq X^1\supseteq X^2\supseteq \dots \supseteq X^M\,,
\end{equation*}
for each $\mathbf{h}=\{h_0,h_1,\dots,h_m\}$, the $\mathbf{h}$-face posets of simplicial complexes $X^\bullet$ form a cofiltration of finite posets:
\begin{equation*}
    P^{\mathbf{h}}_{X^\bullet}\colon P_{X^0}^\mathbf{h}\supseteq P_{X^1}^\mathbf{h}\supseteq P_{X^2}^\mathbf{h}\supseteq \dots \supseteq P_{X^M}^\mathbf{h}.
\end{equation*}
Cofiltration $P^{\mathbf{h}}_{X^\bullet}$ and the corresponding linear maps $\varphi^{n,\bullet,\mathbf{h}}$ from equation~\eqref{eq:holesmap} induce a ladder module: 

\begin{equation} \label{Eq:Diagram6pack}
    \begin{tikzcd}
 H^n(P_{X^0}^\mathbf{h};\Bbbk)  \arrow[r] & H^n(P_{X^1}^\mathbf{h};\Bbbk)  \arrow[r] & H^n(P_{X^2}^\mathbf{h};\Bbbk)  \arrow[r] & \dots \arrow[r] & H^n(P_{X^M}^\mathbf{h};\Bbbk)  \\
 H^n(X^0;\Bbbk) \arrow[u, "\varphi^{n,0,\mathbf{h}}"] \arrow[r] & H^n(X^1;\Bbbk) \arrow[u, "\varphi^{n,1,\mathbf{h}}"] \arrow[r] & H^n(X^2;\Bbbk) \arrow[u,"\varphi^{n,2,\mathbf{h}}"] \arrow[r] & \dots \arrow[r] & H^n(X^M;\Bbbk) \arrow[u,"\varphi^{n,M,\mathbf{h}}"]
    \end{tikzcd}
\end{equation}
Similar to the previous section, from diagram~\eqref{Eq:Diagram6pack} we get the following collection of one-parameter persistence modules:

\begin{itemize}
    \item The $n$-th persistence module of the cofiltration, on the bottom horizontal line,
    \begin{equation*}
        H^n(X^\bullet;\Bbbk)\,\colon \ H^n(X^0;\Bbbk)\to H^n(X^1;\Bbbk)\to \dots \to H^n(X^M;\Bbbk)\to 0.
    \end{equation*}
    \item The $n$-th persistence module of the cofiltration of $\mathbf{h}$-face posets, $P_{X^\bullet}^\mathbf{h}$, on the top horizontal line,
    \begin{equation*}
        H^n(P_{X^\bullet}^\mathbf{h};\Bbbk)\,\colon \ H^n(P_{X^0}^\mathbf{h};\Bbbk)\to H^n(P_{X^1}^\mathbf{h};\Bbbk)\to \dots \to H^n(P_{X^M}^\mathbf{h};\Bbbk)\to 0.
    \end{equation*}
    \item The image persistence module 
    \begin{equation*}
        \img \varphi^{n,\bullet,\mathbf{h}}\,\colon \ \img\varphi^{n,0,\mathbf{h}}\to \img\varphi^{n,1,\mathbf{h}}\to \dots \to \img\varphi^{n,M,\mathbf{h}}\to 0.
    \end{equation*}
    It tracks the evolution of the holes in the cofiltration $X^\bullet$ defined by simplices whose dimension is in $\mathbf{h}$.
    
    \item The kernel persistence module  
    \begin{equation*}
        \ker \varphi^{n,\bullet,\mathbf{h}}\,\colon \ \ker\varphi^{n,0,\mathbf{h}}\to \ker\varphi^{n,1,\mathbf{h}}\to \dots \to \ker\varphi^{n,M,\mathbf{h}}\to 0.
    \end{equation*}
    It records the evolution throughout the cofiltration $X^\bullet$ of the holes that disappear when removing simplices whose dimension is not in $\mathbf{h}$ (the less cohesive ones).
    \item The cokernel persistence module 
    \begin{equation*}
         \coker \varphi^{n,\bullet,\mathbf{h}}\,\colon \ \coker\varphi^{n,0,\mathbf{h}}\to \coker\varphi^{n,1,\mathbf{h}}\to \dots \to \coker\varphi^{n,M,\mathbf{h}}\to 0.
    \end{equation*}
    It contains the dynamics of the holes that are created by the elimination of simplices. 
\end{itemize}

\begin{example}\label{Ex:6Pack}
Let us continue with the analysis of the cofiltration in Example~\ref{Ex:PersistentThick} as previously pointed out. We adopt the viewpoint of the barycentric subdivision (see equation~\eqref{e:BS}) and reinterpret the degeneration cofiltration presented in Example~\ref{Ex:PersistentThick} accordingly (see Figure~\ref{Fig:FiltrationPersistentCohesive}a).
    
Figure~\ref{Fig:FiltrationPersistentCohesive}d shows the persistence barcodes associated with the $\{0,2\}$-face poset cofiltration. We observe that many bars appear in $H^1(P_{X^{\bullet}}^{\{0,2\}};\Z_2)$. The cofiltration defined by the order complexes of the $\{0,2\}$-face posets (Figure~\ref{Fig:FiltrationPersistentCohesive}b) illustrates the persistence of such bars. Moreover, comparing with $\mathcal{K}(X^\bullet)$ (Figure~\ref{Fig:FiltrationPersistentCohesive}a), we can see that most of the bars in $H^1(P_{X^{\bullet}}^{\{0,2\}};\Z_2)$ correspond to holes that are created by edge removal. From an algebraic point of view, the persistence of such holes is given by the cokernel persistence barcode $\coker \varphi^{1,\bullet,\{0,2\}}$ (Figure~\ref{Fig:FiltrationPersistentCohesive}h). To focus on the persistence of the hole in $X^\bullet$ when edges are removed, we study the image persistence module $\img \varphi^{1,\bullet,\{0,2\}}$ (Figure~\ref{Fig:FiltrationPersistentCohesive}f). As its barcode shows, in this example it captures the same information as the $2$-coskeleton cofiltration barcode (Figure~\ref{Fig:FiltrationPersistentCosk}d), reflecting that the hole in $X^\bullet$ is no longer fully surrounded by triangles in stage 4.

 On the other hand, the persistence barcode $\img \varphi^{1,\bullet,\{1,2\}}$ (Figure~\ref{Fig:FiltrationPersistentCohesive}f) consists of the interval $[0,1)$, meaning that the hole in $X^\bullet$ is only enclosed by a closed $(1,2)$-walk in $X^0$. Thus, the study of the cohesion of the network detects that the deterioration affects the hole from the very first stage. Further information is provided by the intervals in $\coker \varphi^{0,\bullet,\{1,2\}}$ (Figure~\ref{Fig:FiltrationPersistentCohesive}i). When we remove the vertices, two $(1,2)$-connected components are born in step $2$ and one more is born in step $3$, although neither connectivity nor $(0,2)$-connectivity are affected along the original cofiltration (Figure~\ref{Fig:FiltrationPersistentCosk}c and Figure~\ref{Fig:FiltrationPersistentCohesive}d). This points to an increase of critical vertices in the network. 
\end{example}

\begin{figure}[!htb]
    \centering
    \begin{subfigure}{0.9\textwidth}
        \includegraphics[width=\textwidth,valign=t]{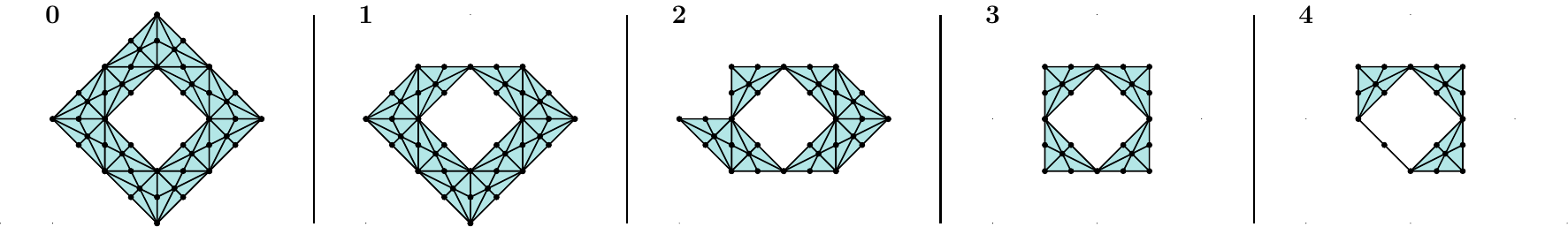}
        \caption{Cofiltration $\mathcal{K}(X^\bullet)$}\label{Fig:FiltrationPersistentCohesiveA}
    \end{subfigure}

    \begin{subfigure}{0.9\textwidth}
        \includegraphics[width=\textwidth,valign=t]{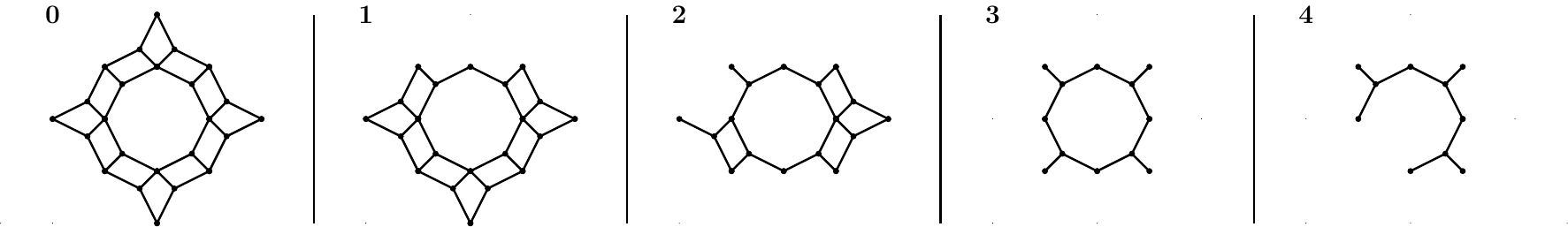}
        \caption{Cofiltration $\mathcal{K}(P_{X^\bullet}^{\{0,2\}})$}\label{Fig:FiltrationPersistentCohesiveB}
    \end{subfigure}
    
    \begin{subfigure}{0.9\textwidth}
        \includegraphics[width=\textwidth,valign=t]{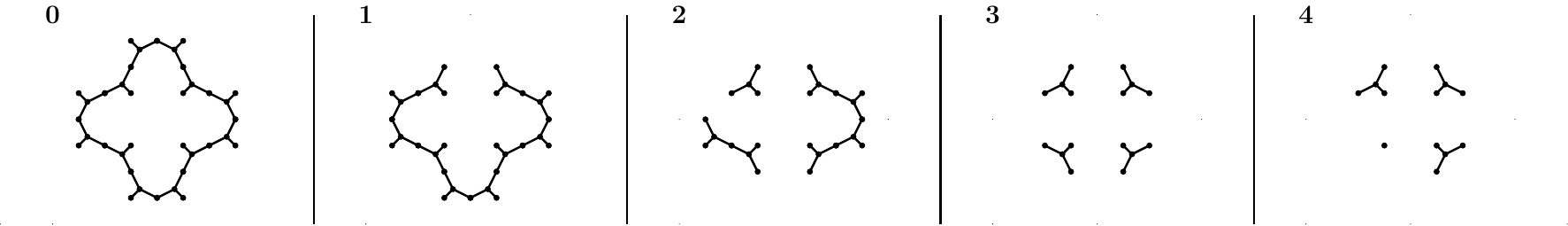}
        \caption{Cofiltration $\mathcal{K}(P_{X^\bullet}^{\{1,2\}})$}\label{Fig:FiltrationPersistentCohesiveC}
    \end{subfigure}

    \begin{minipage}[b]{0.3\textwidth}
        \centering
        \includegraphics[width=\textwidth,valign=t]{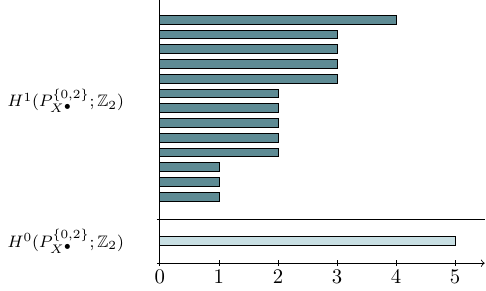}
        \subcaption{\small Barcodes of the $\{0,2\}$-face poset cofiltration}\label{Fig:FiltrationPersistentCohesiveD}
    \end{minipage}
    \hfill
    \begin{minipage}[b]{0.3\textwidth}
        \centering
        \includegraphics[width=\textwidth,valign=t]{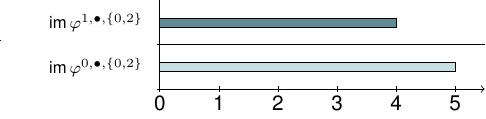}
        \subcaption{\small Barcodes of the $\{0,2\}$-image persistence modules}\label{Fig:FiltrationPersistentCohesiveF}
    \end{minipage}
    \hfill
    \begin{minipage}[b]{0.3\textwidth}
        \centering
        \includegraphics[width=\textwidth,valign=t]{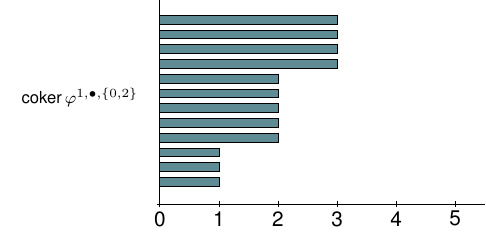}
        \subcaption{\small Barcode of the $\{0,2\}$-cokernel persistence modules}
    \end{minipage}

    \vspace{1em} 
    
    \begin{minipage}[b]{0.3\textwidth}
        \centering
        \includegraphics[width=\textwidth,valign=t]{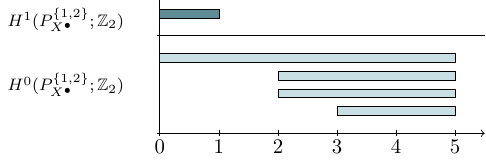}
        \subcaption{\small Barcodes of the $\{1,2\}$-face poset cofiltration}\label{Fig:FiltrationPersistentCohesiveE}
    \end{minipage}
    \hfill
    \begin{minipage}[b]{0.3\textwidth}
        \centering
        \includegraphics[width=\textwidth,valign=t]{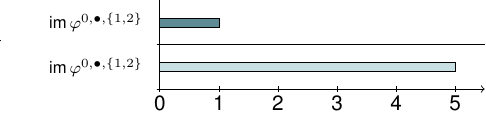}
        \subcaption{\small Barcodes of the $\{1,2\}$-image persistence modules}\label{Fig:FiltrationPersistentCohesiveG}
    \end{minipage}
    \hfill
    \begin{minipage}[b]{0.3\textwidth}
        \centering
        \includegraphics[width=\textwidth,valign=t]{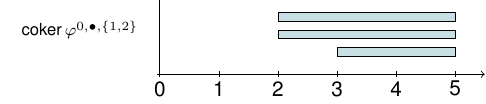}
        \subcaption{\small Barcode of the $\{1,2\}$-cokernel persistence modules}
    \end{minipage}
    \caption{Barycentric subdivision cofiltrations and persistence barcodes}\label{Fig:FiltrationPersistentCohesive}
\end{figure}

\begin{remark}\label{Rem:Software6pack}
\cite{CohenSteiner_Otros09} developed algorithms to compute the barcodes of the image, kernel, and cokernel persistence modules arising from an inclusion of simplicial filtrations. However, public implementations remain scarce. Among the available tools,~\cite{Bauer_Schmahl23} integrated an algorithm for image persistence into the Ripser framework~\cite{Bauer21}, enabling efficient barcode computation for an inclusion of Vietoris-Rips filtrations. In addition, the \texttt{PerMaViss} package~\cite{Torras20} allows for the computation of images and kernels of morphisms of persistence modules. As another example,~\cite{Cultera_Otros24} provide implementations to compute image, kernel, and cokernel persistence diagrams associated with chromatic alpha complexes.
\end{remark}

\section{Stability results}\label{Sec:Stability}

Given a point cloud or, more generally, a network~\cite{Memoli_Chowdhury18, Carlsson_Otros14}, a standard strategy in TDA is to capture its topological structure using the Vietoris-Rips filtration. In the present setting, it is convenient to use the equivalent cofiltration viewpoint. Let $\mathfrak{N}$ denote the collection of all networks, where a network $(X,\omega_X)$ consists of a finite set $X$ equipped with a function $\omega_X\colon X\times X\to \mathbb{R}$. For each $\varepsilon\in \mathbb{R}$, we define
\[
\mathcal{VR}^\varepsilon(X)\coloneqq \{\emptyset\neq\sigma\subseteq X : \max_{x,x'\in \sigma}\omega_X(x,x')\leq -\varepsilon\}\,.
\]
The negative parameter turns the usual growing Vietoris-Rips filtration into a decreasing process, which we interpret as progressive degradation. The goal of this section is to study the stability of the thickness and cohesiveness resilience analysis for Vietoris-Rips cofiltrations associated with finite networks.

The persistence modules arising from Vietoris-Rips cofiltrations will be compared using the bottleneck distance, denoted by $d_B$~\cite{Chazal_Otros16}. To quantify the dissimilarity between the underlying networks, a standard approach is to use the network distance, which naturally generalizes the classical Gromov-Hausdorff distance to the setting of arbitrary weight functions~\cite{Memoli_Chowdhury18}. 

\begin{definition}
    Let $(X,\omega_X),\,(Y,\omega_Y)\in\mathfrak{N}$. A \emph{correspondence} between $X$ and $Y$ is a subset $C\subseteq X\times Y$ such that $\pi_X(C)=X$ and $\pi_Y(C)=Y$, where $\pi_X\colon X\times Y\to X$ and $\pi_Y\colon X\times Y\to Y$ are the natural projections. The set of correspondences between $X$ and $Y$ will be denoted by $\mathfrak{C}(X,Y)$. The \emph{distortion} of a correspondence $C\in \mathfrak{C}(X,Y)$ is
\[
\operatorname{dis}(C)\coloneqq \max_{(x,y),(x',y')\in C} \left| \omega_X(x,x')-\omega_Y(y,y')\right| \,.
\]
The \emph{network distance} $d_{\mathfrak{N}}\colon \mathfrak{N}\times \mathfrak{N}\to \mathbb{R}_{\geq 0}$ is defined as: 
\[
d_{\mathfrak{N}}(X,Y)\coloneqq \frac{1}{2}\min_{C\in\mathfrak{C}(X,Y)}\operatorname{dis}(C)\,.
\]
\end{definition}

\subsection{Stability theorems for the study of thickness}

Given a network $(X,\omega_X)$, the $q$-coskeleta of the complexes in the Vietoris-Rips cofiltration of $(X,\omega_X)$ form a cofiltration, $\mathcal{VR}^\bullet(X)^q$. We now study the stability of the persistent cohomology modules associated with these coskeletal cofiltrations:
\[
\{H^n(\mathcal{VR}^\bullet(X)^q;\Bbbk)\}_{n\geq 0}.
\]
A first observation is that the persistent cohomology of the cofiltration of $q$-coskeleta of Vietoris-Rips complexes is not stable, in general, with respect to the network distance, as the following example shows.

\begin{example}\label{Ex:InestabilidadGHGruesos}
Let $(X,\omega_X)$ be a network consisting of a single point $x$ such that $\omega_X(x,x)=0$, and let $(Y,\omega_Y)$ be a network consisting of two points, $y,y'$, such that $\omega_Y(y,y)=\omega_Y(y',y')=0$ and $\omega_Y(y,y')=\omega_Y(y',y)=2$. The only correspondence between $X$ and $Y$ is $C=\{(x,y),(x,y')\}$, and therefore
\[
d_{\mathfrak{N}}(X,Y)=\frac{1}{2}\operatorname{dis}(C)=\frac{1}{2}\left|\omega_X(x,x)-\omega_Y(y,y')\right|=1\, .
\]
The Vietoris-Rips complexes $\mathcal{VR}^\varepsilon(X)$ consist of a single point for every $\varepsilon\leq 0$. Hence, their $1$-coskeleta are all empty and $H^0(\mathcal{VR}^\bullet(X)^1;\Bbbk)$ is the zero persistence module.

On the other hand, the Vietoris-Rips complexes $\mathcal{VR}^\varepsilon(Y)$ consist of two isolated vertices for $-2<\varepsilon \leq 0$ and of an edge for every $\varepsilon\leq -2$. Thus, the cofiltration of $1$-coskeleta $\mathcal{VR}^\bullet(Y)^1$ consists of a single edge for $\varepsilon\leq -2$ and is empty for $\varepsilon>-2$. Consequently,
\[
H^0(\mathcal{VR}^\bullet(Y)^1;\Bbbk)=\Bbbk(-\infty,-2]\,.
\]
The bottleneck distance between both persistence modules is infinite, and therefore larger than the distance between the original networks:
\[
d_B(H^0(\mathcal{VR}^\bullet(X)^1;\Bbbk),H^0(\mathcal{VR}^\bullet(Y)^1;\Bbbk))=\infty > 1 =d_{\mathfrak{N}}(X,Y)\, .
\]
\end{example}

To establish a stability theorem for coskeletal cofiltrations, we require an alternative notion of distance. Since the coskeletal cofiltration is sensitive to the total number of vertices of the network, as shown by the previous example, we restrict the comparison to networks of the same cardinality. To this end, we use the \emph{second network distance}~\cite{Chodhury_Memoli23}, defined as follows:

\begin{definition}\index{Second network distance}
Given two networks $(X,\omega_X)$ and $(Y,\omega_Y)$ with the same cardinality, define
\[
\widehat d_{\mathfrak{N}}(X,Y)\coloneqq \frac{1}{2}\min_{\phi}\max_{x,x'\in X}\left| \omega_X(x,x')-\omega_Y(\phi(x),\phi(x'))\right| \, ,
\]
where $\phi\colon X\to Y$ ranges over all bijections from $X$ to $Y$.
\end{definition}

\begin{remark}
To simplify the notation, in what follows we omit the coefficient field $\Bbbk$ with respect to which cohomology is taken, writing $H^n\mathcal{VR}^\bullet(X)^q$ instead of $H^n(\mathcal{VR}^\bullet(X)^q;\Bbbk)$.
\end{remark}

\begin{theorem}\label{Thm:EstabilidadGruesos}
Let $(X,\omega_X)$ and $(Y,\omega_Y)$ be two networks with the same cardinality. For every $q\geq 0$ and every $n\geq 0$, one has 
\[
d_{B}\left(H^n\mathcal{VR}^\bullet(X)^q,H^n\mathcal{VR}^\bullet(Y)^q\right)\leq 2\,\widehat d_{\mathfrak{N}}(X,Y)\, .
\]
\end{theorem}
\begin{proof}
The proof is deferred to Appendix~\ref{proof:Thm-EstabilidadGruesos}.
\end{proof}

On the other hand, the inclusions $\mathcal{VR}^\varepsilon(X)^q\subseteq \mathcal{VR}^\varepsilon(X)$ induce, for each $n\geq 0$, a morphism of persistence modules
\[
\varphi^{n,\bullet,q}\colon H^n\mathcal{VR}^\bullet(X)\longrightarrow H^n\mathcal{VR}^\bullet(X)^q\,.
\]
The following result guarantees the stability of the persistence modules defined by the images, kernels, and cokernels of $\varphi^{n,\bullet,q}$.

\begin{theorem}\label{Thm:EstabilidadPersistenciaGrosor}
Let $(X,\omega_X),(Y,\omega_Y)\in \mathfrak{N}$ be networks with the same cardinality. Then the following inequalities hold:
\[
\begin{gathered}
    d_B\left(\img\varphi^{n,\bullet,q}_X,\,\img\varphi_Y^{n,\bullet,q}\right) \leq 2\,\widehat d_{\mathfrak{N}}(X,Y)\,,\\[1.5ex]
    d_B\left(\ker\varphi^{n,\bullet,q}_X,\,\ker\varphi_Y^{n,\bullet,q}\right) \leq 2\,\widehat d_{\mathfrak{N}}(X,Y)\,,\\[1.5ex]
    d_B\left(\coker\varphi^{n,\bullet,q}_X,\,\coker\varphi_Y^{n,\bullet,q}\right) \leq 2\,\widehat d_{\mathfrak{N}}(X,Y)\, .
\end{gathered}
\]
\end{theorem}
\begin{proof}
The proof is given in Appendix~\ref{proof:Thm-EstabilidadPersistenciaGrosor}.
\end{proof}

\subsection{Stability theorems for the study of cohesiveness}

Given a network $(X,\omega_X)$, the $\mathbf{h}$-face posets of the complexes in the Vietoris-Rips cofiltration $\mathcal{VR}^\bullet(X)$ define a cofiltration of finite posets indexed by $\mathbb{R}$, $\{P_{\mathcal{VR}^\varepsilon(X)}^{\mathbf{h}}\}_{\varepsilon\in \mathbb{R}}$. By composing with the $n$-th poset cohomology functor with coefficients in $\Bbbk$, we obtain a p.f.d.\ persistence module which will write as $H^n P_{\mathcal{VR}^\bullet(X)}^{\mathbf{h}}$. In general, however, this persistence module is not stable with respect to the network distance, as the following example illustrates. 

\begin{example}
Consider the networks $(X,\omega_X)$ and $(Y,\omega_Y)$ from Example~\ref{Ex:InestabilidadGHGruesos}. The cofiltration $P_{\mathcal{VR}^\bullet(X)}^{\{1\}}$ is empty, while $P_{\mathcal{VR}^\bullet(Y)}^{\{1\}}$ consists of a single point for $\varepsilon\leq -2$ and is empty for $\varepsilon>-2$. Hence,
\[
H^0 P_{\mathcal{VR}^{\bullet}(X)}^{\{1\}}=0\quad \text{ and }\quad H^0P_{\mathcal{VR}^\bullet(Y)}^{\{1\}}\simeq\Bbbk(-\infty,-2]\,.
\]
Therefore, the bottleneck distance between the $\mathbf{h}$-cohesiveness persistence modules exceeds the distance between the networks:
\[
d_B\big(H^0 P_{\mathcal{VR}^{\bullet}(X)}^{\{1\}},\,H^0P_{\mathcal{VR}^{\bullet}(Y)}^{\{1\}}\big)=\infty> 1 =d_{\mathfrak{N}}(X,Y)\, .
\]
\end{example}

As in the case of thickness persistence modules, in order to prove stability theorems for cohesiveness persistence modules, we restrict the comparison to networks with the same cardinality and use the second network distance.

\begin{theorem}\label{Thm:EstabilidadCohesion}
Let $(X,\omega_X)$ and $(Y,\omega_Y)$ be two networks with the same cardinality. For every $\mathbf{h}$ and every $n\geq 0$, one has
\[
d_{B}\big(H^nP_{\mathcal{VR}^\bullet(X)}^{\mathbf{h}},\,H^nP_{\mathcal{VR}^\bullet(Y)}^{\mathbf{h}}\big)\leq 2\,\widehat d_\mathfrak{N}(X,Y)\,.
\]
\end{theorem}
\begin{proof}
The proof is given in Appendix~\ref{proof:Thm-EstabilidadCohesion}. 
\end{proof}

As shown by the following result, stability also holds for the image, kernel, and cokernel persistence modules induced by the morphisms of persistence modules
\[
\varphi^{n,\bullet,\mathbf{h}}\colon H^n P_{\mathcal{VR}^{\bullet}(X)} \longrightarrow H^n P_{\mathcal{VR}^{\bullet}(X)}^{\mathbf{h}}\, ,
\]
defined by the inclusions $P_{\mathcal{VR}^\varepsilon(X)}^{\mathbf{h}}\subseteq P_{\mathcal{VR}^\varepsilon(X)}$.

\begin{theorem}\label{Thm:EstabilidadPersistenciaCohesion}
Let $(X,\omega_X),(Y,\omega_Y)\in \mathfrak{N}$ be networks with the same cardinality. Then the following inequalities hold:
\[
\begin{gathered}
    d_B\big(\img\varphi^{n,\bullet,\mathbf{h}}_X,\img\varphi_Y^{n,\bullet,\mathbf{h}}\big) \leq 2\,\widehat d_{\mathfrak{N}}(X,Y)\,,\\[1.5ex]
    d_B\big(\ker\varphi^{n,\bullet,\mathbf{h}}_X,\ker\varphi_Y^{n,\bullet,\mathbf{h}}\big) \leq 2\,\widehat d_{\mathfrak{N}}(X,Y)\,,\\[1.5ex]
    d_B\big(\coker\varphi^{n,\bullet,\mathbf{h}}_X,\coker\varphi_Y^{n,\bullet,\mathbf{h}}\big) \leq 2\,\widehat d_{\mathfrak{N}}(X,Y)\, .
\end{gathered}
\]
\end{theorem}
\begin{proof}
The proof is given in Appendix~\ref{proof:Thm_EstabilidadPersistenciaCohesion}. 
\end{proof}

\section{Conclusion}

In this work, we have introduced a persistence framework for analyzing robustness in finite simplicial complexes beyond ordinary Betti numbers. The central idea is that, in higher-order networks, the relevance of a homological feature is not exhausted by its existence or by its persistence along a filtration. It is also important to understand how that feature is supported by the simplicial structure: whether it is sustained by high-dimensional simplices, and whether the simplices involved remain connected through strong higher-order adjacencies. The examples in Figure~\ref{Fig:Fig1Summary} summarize this point: the complexes have the same homotopy type, but the refined invariants introduced here distinguish structural differences that ordinary Betti numbers do not detect.

\begin{figure}[htb!]
\centering
\includegraphics[width=0.9\linewidth]{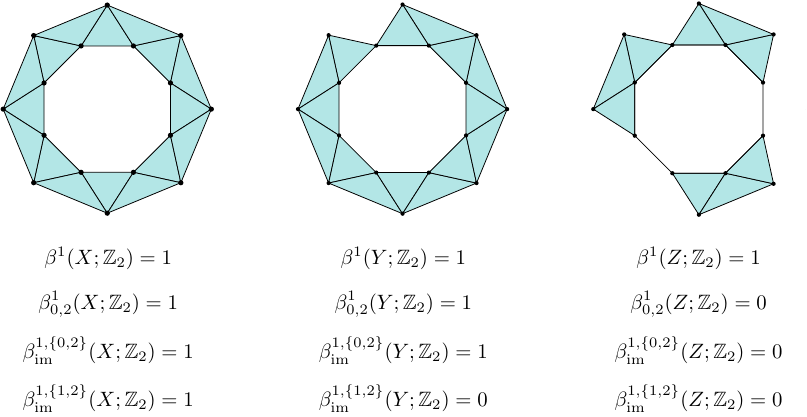}
\caption{Simplicial complexes with the same homotopy type but different higher-order support structures. The refined Betti numbers introduced in this work distinguish structural features that are invisible to ordinary Betti numbers} \label{Fig:Fig1Summary}
\end{figure}

The first family of invariants, the thick Betti numbers, measures the extent to which connected components and holes are supported by simplices of sufficiently high dimension. This is achieved through the cofiltration by coskeleta, which progressively removes those simplices that are not supported by higher-dimensional cofaces. The second family, the cohesive Betti numbers, measures how cohomology classes depend on higher-order adjacencies among simplices of prescribed dimensions. This construction is naturally formulated using dimension-selected face posets. Together, these two families provide complementary ways of refining classical topological information: thickness records the dimensional quality of the support, while cohesiveness records the strength of the higher-order connections sustaining it.

We have also shown how these invariants can be incorporated into a persistence study of degradation processes. Given an attack cofiltration representing a prescribed removal of simplices, one can combine the attack parameter with either the thickness or the cohesiveness parameter, which produces biparameter persistence modules. To extract interpretable one-parameter information from them, we fix one structural parameter and study the resulting ladder modules. Their associated image, kernel, and cokernel persistence modules separate features that remain thick or cohesive during the process, features that lose thickness or cohesion before disappearing, and features that are created by the degradation itself. This provides a precise persistent interpretation of resilience for higher-order cohomological features.

A further contribution is the stability analysis of the constructions. We identify an obstruction to stability with respect to the classical network distance. For networks with the same number of vertices, however, this difficulty is overcome by using the second network distance, for which we establish stability results. This provides a theoretical reliability guarantee for the proposed constructions.

The framework developed here is primarily mathematical, but it is motivated by contexts in which higher-order homological features have already proved informative. Persistent loops and voids have been used to distinguish cancer-specific gene regulatory networks from normal ones~\cite{Masoomy_Otros21}, and recent work relates $\beta_2$-structures in cancer consensus networks to driver or cancer-associated genes~\cite{Ramos_Otros25}. In neuroscience, persistent holes have been proposed as markers of abnormal metabolic connectivity in Alzheimer's disease~\cite{Lee_Chung_Kang_Lee14}, and homological scaffolds have been used to describe mesoscopic cycles in functional brain networks~\cite{Petri_Otros14}. These works suggest that future applications may benefit from evaluating not only the persistence of homological features, but also the robustness of their higher-order support. The thick and cohesive Betti numbers introduced here provide a theoretical basis for this analysis, while domain-specific pipelines and empirical validation are left for future work.

\section*{Acknowledgments} 
The authors would like to thank the referees for their careful reading of the manuscript. Their feedback has helped us to clarify several aspects of the paper, improve the exposition, and better contextualize the proposed constructions. The authors would also like to thank Carles Casacuberta, Rub\'en Ballester, and Aina Ferr\`a for their valuable comments and suggestions. This work is supported by Spanish National Grant PID2021-128665NB-I00 funded by MCIN/AEI/10.13039/501100011033 and, as appropriate, by ``ERDF A way of making Europe''; and also by project STAMGAD 18.J445/463AC03 by Consejer\'ia de Educaci\'on (GIR, Junta de Castilla y Le\'on, Spain).  Pablo Hern\'andez-Garc\'ia also acknowledges financial support from the USAL 2022 call for predoctoral contracts, co-financed by Banco Santander.

\bibliographystyle{siam}
\bibliography{bibliography.bib}

\appendix
\section{}\label{app:demos}

We give in this appendix the proofs of several results stated throughout the paper. First, let us recall that any simplicial map $f\colon X\to Y$ induces linear maps  
\begin{equation*}
        {f^n}\colon C^n(Y;\Bbbk)\to C^n(X;\Bbbk),
    \end{equation*}
 defined, for each $y\in C^n(Y;\Bbbk)$ and each ordered $n$-simplex $\sigma=({i_0},\dots,{i_n})\in X$, as
\begin{equation}\label{Eq:DefMorfismosEnCohomologia}
{f^n}(y)_{({i_0},\dots,{i_n})}= \begin{cases}
(-1)^{|\epsilon_\sigma|}y_{\epsilon_\sigma\cdot f(\sigma)} & \text{ if }f({i_j})\neq f({i_k})\, \forall\, i_j\neq i_k,\\
0 & \text{ otherwise,}
\end{cases}
\end{equation}
where $\epsilon_\sigma$ is the permutation on the indices $\{0,\dots, n\}$ such that (the indices of) the vertices of $f(\sigma)$ appear according to the ordering in $Y$, that is, and abusing notation, $f({i_{\epsilon(0)}})<\cdots< f({i_{\epsilon(n)}})$. We denote by $\epsilon_\sigma\cdot f(\sigma)$ the $n$-simplex $(f({i_{\epsilon(0)}}),\dots, f({i_{\epsilon(n)}}))$ and $(-1)^{|\epsilon_\sigma|}$ denotes the sign of $\epsilon_\sigma$.

These linear maps satisfy ${f^{n+1}}\circ \delta^n=\delta^n\circ {f^n}$, where $\delta^n\colon C^n(X;\Bbbk)\to C^{n+1}(X;\Bbbk)$ is the coboundary operator, so they induce morphisms between the cohomology spaces
\begin{equation*}
\begin{array}{rccc}
H^n(f)\colon & H^n(Y;\Bbbk)&\longrightarrow & H^n(X;\Bbbk)\\
& [y] & \longmapsto & [{f^n}(y)]
\end{array}
\end{equation*}
The above assignments are functorial, which means that they are compatible with identity and composition: $H^n(\mathrm{Id})=\mathrm{Id}$ and $H^n(g\circ f)=H^n(f)\circ H^n(g)$.

\subsection{Proofs of Section~\ref{Sec:Coskeletal}}\label{app:Sec-Coskeletal}

\subsubsection{Proof of Proposition~\ref{Prop:Coskeletal0h}}\label{proof:Prop-Coskeletal0h}
Recall that $\beta^{0,q}(X;\Bbbk)$ is the number of connected components of $X^q$. That is, 
        \begin{equation}\label{Eq:XhCompConexas}
            \beta^{0,q}(X;\Bbbk)=\#\faktor{V(X^q)}{\sim_{01}}.
        \end{equation}
        Given a walk $\{e_1,e_2,\dots,e_r\}$ of $0$-lower adjacent $1$-simplices in $X^q$, by definition of the $q$-coskeleton, there exists a $(0,q)$-walk $\{\sigma_1,\sigma_2,\dots,\sigma_r\}$ in $X^q$ such that $e_i\face \sigma_i$ for each $i$. Consequently, \begin{equation*}
            \#\faktor{V(X^q)}{\sim_{01}}=\#\faktor{V(X^q)}{\sim_{0q}}.
        \end{equation*}
        The equivalence classes of the last quotient correspond to the non-singleton equivalence classes of $V(X)/\sim_{0q}$, which yields the first formula in the proposition. 

        To establish the second formula, by equation~\eqref{Eq:XhCompConexas}, it is enough to construct a bijection between $V(X^q)/\sim_{01}$ and $S^{\geq q}(X)/\sim_{0q}$. For any $[v]\in V(X^q)/\sim_{01}$, the definition of the $q$-coskeleton ensures the existence of a simplex $\sigma\in S^{\geq q}(X)$ such that $v\in \sigma$. Following this notation, we define the map:
        \[
        \begin{array}{ccc}
            \faktor{V(X^q)}{\sim_{01}} &\longrightarrow & \faktor{S^{\geq q}(X)}{\sim_{0q}}\\[1.5ex]
            {} [v] &\longmapsto & {} [\sigma]
        \end{array}
        \]
        This assignment is well-defined, as it is independent of both the chosen representative $v \in [v]$ and the simplex $\sigma$ containing $v$.

        Conversely, given $[\sigma]\in S^{\geq q}(X)/\sim_{0q}$, let $v$ be a vertex of $\sigma$. Then $v\in X^q$, and we define the map:
        \[
        \begin{array}{ccc}
            \faktor{S^{\geq q}(X)}{\sim_{0q}}& \longrightarrow &  \faktor{V(X^q)}{\sim_{01}}\\[1.5ex]
            {} [\sigma] & \longmapsto & {} [v]
        \end{array}
        \]
        This map is also well-defined, depending neither on the representative $\sigma \in [\sigma]$ nor on the choice of $v\in \sigma$. These two maps are mutually inverse, establishing the desired bijection.
    
\subsubsection{Proof of Theorem~\ref{Prop:InvarianzaFunctoresCoesq}} 
\label{proof:Prop-InvarianzaFunctoresCoesq}
Let $f\colon X\stackrel{\sim}\to Y$ be a simplicial isomorphism. For each $q\geq 0$, $f$ restricts to an isomorphism $f^q\colon X^q\stackrel{\sim}\to Y^q$. Moreover, for each $q\leq q'$, these isomorphisms make the following diagram commutative:
\begin{equation*}
\begin{tikzcd}
X^{q'} \arrow[r, "\sim"] \arrow[d, phantom, sloped, "\subseteq"]& Y^{q'} \arrow[d, phantom, sloped, "\subseteq"]\\
X^{q} \arrow[r, "\sim"] & Y^q
\end{tikzcd}
\end{equation*}
Taking simplicial cohomology, we obtain a commutative diagram:
\begin{equation*}
\begin{tikzcd}
H^n(X^{q'};\Bbbk)  & H^n(Y^{q'};\Bbbk) \arrow[l,"\sim"'] \\
H^n(X^{q};\Bbbk) \arrow[u] & H^n(Y^{q};\Bbbk) \arrow[l, "\sim"'] \arrow[u]
\end{tikzcd}
\end{equation*}
This establishes the desired isomorphism between $H^n(X^\bullet;\Bbbk)$ and $H^n(Y^\bullet;\Bbbk)$.

\subsubsection{Proof of Proposition \ref{Prop:aibiValues}}\label{proof:Prop-aibiValues}
On the one hand, $H^n(X^q;\Bbbk)$ is defined by the coboundary maps
\begin{equation}\label{Eq:CobordeXh}
C^{n-1}(X^{q};\Bbbk)\stackrel{\delta^{n-1}}\longrightarrow C^{n}(X^{q};\Bbbk)\stackrel{\delta^n}\longrightarrow C^{n+1}(X^{q};\Bbbk).
\end{equation}
On the other hand, $H^n(X;\Bbbk)$ is defined by the coboundary maps
\begin{equation}\label{Eq:CobordeX}
C^{n-1}(X;\Bbbk) \stackrel{\delta^{n-1}}\longrightarrow C^{n}(X;\Bbbk) \stackrel{\delta^n}\longrightarrow C^{n+1}(X;\Bbbk).
\end{equation}
By definition, the sequences \eqref{Eq:CobordeXh} and \eqref{Eq:CobordeX} agree when $q<n$, so $H^n(X^q;\Bbbk)\to H^n(X;\Bbbk)$ is the identity. In the case $q=n$, the natural inclusion $i\colon X^{n}\hookrightarrow X$ induces a commutative diagram (equation~\eqref{Eq:DefMorfismosEnCohomologia}): 
\begin{equation*}
\begin{tikzcd}
C^{n-1}(X;\Bbbk)\arrow[r,"\delta^{n-1}"]\arrow[d,twoheadrightarrow,"{i^{n-1}}"] & C^n(X;\Bbbk) \arrow[r,"\delta^n"]\arrow[d,equal,"{i^n}"] & C^{n+1}(X;\Bbbk)\arrow[d,equal,"{i^{n+1}}"] \\
C^{n-1}(X^{n};\Bbbk)\arrow[r,"\delta^{n-1}"] & C^n(X^{n};\Bbbk) \arrow[r,"\delta^n"] & C^{n+1}(X^{n};\Bbbk)
\end{tikzcd}
\end{equation*}
Moreover, $\delta^{n-1}\colon C^{n-1}(X;\Bbbk)\longrightarrow C^n(X;\Bbbk)$ vanishes on $C^{n-1}(X;\Bbbk)\smallsetminus C^{n-1}(X^n;\Bbbk)$. Therefore, $\img\,\delta^{n-1}_X=\img\,\delta^{n-1}_{X^n}$ and hence $H^n(X;\Bbbk)= H^n(X^n;\Bbbk)$, with $H^n(i)$ being the identity map.

For the second part of the proposition, we have the commutative diagram:
\begin{equation*}
\begin{tikzcd}
C^{n-1}(X;\Bbbk)\arrow[r,"\delta^{n-1}"]\arrow[d,twoheadrightarrow,"{i^{n-1}}"] & C^n(X;\Bbbk) \arrow[r,"\delta^n"]\arrow[d,twoheadrightarrow,"{i^n}"] & C^{n+1}(X;\Bbbk)\arrow[d,equal,"{i^{n+1}}"] \\
C^{n-1}(X^{n+1};\Bbbk)\arrow[r,"\delta^{n-1}"] & C^n(X^{n+1};\Bbbk) \arrow[r,"\delta^n"] & C^{n+1}(X^{n+1};\Bbbk)
\end{tikzcd}
\end{equation*}
where the linear map ${i^n}$ is surjective and ${i^{n+1}}$ is the identity map. Thus, given $[x]\in H^n(X^{n+1};\Bbbk)$ and taking any $y\in ({i^n})^{-1}(x)$, by the commutativity of the second square, it follows that $y\in \ker\,\delta^n$ and $H^n(i)([y])=[{i^{n}}(y)]=[x]$. Then, we conclude that the induced morphism $H^n(i)\colon H^n(X;\Bbbk)\longrightarrow H^n(X^{n+1};\Bbbk)$ is surjective.

\subsection{Proofs of Section~\ref{Sec:Stratified}}\label{app:Sec-Stratified}

\subsubsection{Proof of Proposition~\ref{Prop:StratifiedCaminos}}\label{proof:Prop-StratifiedCaminos}
The $0$-th cohomology space $H^0(P_X^\mathbf{h};\Bbbk)$ is defined as the kernel of the coboundary map $\delta^0\colon C^0(\mathcal{K}(P_X^\mathbf{h});\Bbbk)\to C^1(\mathcal{K}(P_X^\mathbf{h});\Bbbk)$. Then, 
\begin{equation*}
H^0(P_X^\mathbf{h};\Bbbk)=\{x\in \prod_{\substack{\sigma\in X\\ \dim \,\sigma\in\mathbf{h}}}\Bbbk \ : x_\sigma=x_\tau \ \ \forall\,\sigma\face\tau \in X : \dim\sigma\,,\dim\tau\in \mathbf{h}\}.
\end{equation*}

In particular, if $x\in H^0(P_X^\mathbf{h};\Bbbk)$ it follows that $x_{\sigma}=x_{\sigma'}$ for every $\sigma,\sigma'$ of dimension $h_0$ such that $\sigma,\sigma'\face \tau$ for some simplex $\tau$ of dimension greater than or equal to $h_1$. Indeed, if $\sigma,\sigma'\face \tau$ then there is a sequence $\{\tau_0,\tau_1,\dots,\tau_r\}$ of $h_1$-dimensional faces of $\tau$ and a sequence $\{\sigma_0,\sigma_1,\dots,\sigma_{r+1}\}$ of $h_0$-dimensional faces of $\tau$ such that
\begin{equation*}
\sigma=\sigma_0\face \tau_0\trianglerighteqslant\sigma_1\face \tau_1\trianglerighteqslant \sigma_2\face \dots \face \tau_{r}\trianglerighteqslant \sigma_{r+1}=\sigma'.
\end{equation*}
Therefore, $x_{\sigma}=x_{\sigma_1}=\dots=x_{\sigma_{r+1}}=x_{\sigma'}$. Thus, the natural projection 
\begin{equation*}
\prod_{\substack{\sigma\in X\\ \dim \,\sigma\in\mathbf{h}}}\Bbbk \longrightarrow \prod_{\substack{\sigma\in X\\ \dim \,\sigma=h_0}}\Bbbk
\end{equation*}
induces a linear map $H^0(P_X^\mathbf{h};\Bbbk)\longrightarrow \Gamma^{h_0,h_1}(X;\Bbbk)$. 

Conversely, we can define a linear map $\Gamma^{h_0,h_1}(X;\Bbbk)\longrightarrow H^0(P_X^\mathbf{h};\Bbbk)$ by assigning to each $x\in \Gamma^{h_0,h_1}(X;\Bbbk)$ the vector $x'\in H^0(P_X^\mathbf{h};\Bbbk)$ such that $x'_\tau\coloneqq x_{\sigma}$ for any $h_0$-face $\sigma$ of $\tau$. This map is well defined in the sense that it does not depend on the selected $h_0$-face $\sigma$ of $\tau$, by the definition of $\Gamma^{h_0,h_1}(X;\Bbbk)$. 

The linear maps $H^0(P_X^\mathbf{h};\Bbbk)\longrightarrow \Gamma^{h_0,h_1}(X;\Bbbk)$ and $\Gamma^{h_0,h_1}(X;\Bbbk)\longrightarrow H^0(P_X^\mathbf{h};\Bbbk)$ are mutually inverse, and we conclude.

\subsubsection{Proof of Corollary~\ref{cor:StratifiedCaminos}}\label{proof:cor-StratifiedCaminos}
By Proposition~\ref{Prop:StratifiedCaminos} the dimension of $H^{0}(P_X^\mathbf{h};\Bbbk)$ coincides with the dimension of the space $\Gamma^{h_0,h_1}(X;\Bbbk)$, which corresponds to the number of equivalence classes on $S^{h_0}(X)/\sim_{h_0h_1}$.

\subsection{Proofs of Section~\ref{Sec:Stability}}\label{app:Sec-Stability}

Before proceeding with the proofs of the stability theorems, we briefly recall the definition of the interleaving distance, as well as the Isometry Theorem that relates it to the bottleneck distance, which will be fundamental for our arguments.

\begin{definition}\index{Interleaving distance}\label{Def:Interleaving}
    Given $\varepsilon\geq 0$, two persistence modules $\mathcal{M},\mathcal{N}$ indexed by the reals are said to be \emph{$\varepsilon$-interleaved} if there exist families of linear maps $\{\phi_s\colon \mathcal{M}_s\to \mathcal{N}_{s+\varepsilon}\}_{s\in \mathbb{R}}$ and $\{\psi_s\colon \mathcal{N}_s\to \mathcal{M}_{s+\varepsilon}\}_{s\in \mathbb{R}}$ such that, for every $s\leq t$, the following diagrams commute:
    \[
\begin{gathered}
    \begin{tikzcd}[column sep=8ex, row sep=6ex]
        \mathcal{M}_s \arrow[r,"\mathcal{M}_{s\leq t}"] \arrow[d,"\phi_s"] 
        & \mathcal{M}_t \arrow[d,"\phi_t"] \\
        \mathcal{N}_{s+\varepsilon}\arrow[r, "\mathcal{N}_{s+\varepsilon\leq t+\varepsilon}"] 
        & \mathcal{N}_{t+\varepsilon}
    \end{tikzcd}
    \qquad 
    \begin{tikzcd}[column sep=4ex, row sep=6ex]
        \mathcal{M}_{s-\varepsilon} \arrow[rd, "\phi_{s-\varepsilon}"'] \arrow[rr,"\mathcal{M}_{s-\varepsilon\leq s+\varepsilon}"] 
        & & \mathcal{M}_{s+\varepsilon} \\
        & \mathcal{N}_s \arrow[ru,"\psi_s"'] &
    \end{tikzcd}
    \\[2ex] 
    \begin{tikzcd}[column sep=8ex, row sep=6ex]
        \mathcal{N}_s \arrow[r,"\mathcal{N}_{s\leq t}"] \arrow[d,"\psi_s"] 
        & \mathcal{N}_t \arrow[d,"\psi_t"] \\
        \mathcal{M}_{s+\varepsilon}\arrow[r, "\mathcal{M}_{s+\varepsilon\leq t+\varepsilon}"] 
        & \mathcal{M}_{t+\varepsilon}
    \end{tikzcd}
    \qquad 
    \begin{tikzcd}[column sep=4ex, row sep=6ex]
        \mathcal{N}_{s-\varepsilon} \arrow[rd, "\psi_{s-\varepsilon}"'] \arrow[rr,"\mathcal{N}_{s-\varepsilon\leq s+\varepsilon}"] 
        & & \mathcal{N}_{s+\varepsilon} \\
        & \mathcal{M}_s \arrow[ru,"\phi_s"'] &
    \end{tikzcd}
\end{gathered}
\]
The \emph{interleaving distance} between $\mathcal{M}$ and $\mathcal{N}$ is defined as
\[
d_I(\mathcal{M},\mathcal{N})\coloneqq \inf\{\varepsilon\geq 0 : \mathcal{M}\text{ and }\mathcal{N} \text{ are $\varepsilon$-interleaved}\}\,.
\]
\end{definition}

\begin{theorem}[Isometry Theorem, {\cite[Theorem 3.4]{Lesnick15}}]\label{Thm:isometria}\index{Isometry Theorem}
    The bottleneck and interleaving distances coincide for p.f.d.\ persistence modules indexed by $\mathbb{R}$. That is, if $\mathcal{M}$ and $\mathcal{N}$ are p.f.d.\ persistence modules, then
    \[
    d_B(\mathcal{M},\mathcal{N})=d_I(\mathcal{M},\mathcal{N})\,.
    \]
\end{theorem}

\subsubsection{Proof of Theorem \ref{Thm:EstabilidadGruesos}}
\label{proof:Thm-EstabilidadGruesos}
    Let $\delta=\widehat{d}_{\mathfrak{N}}(X,Y)$ and let $\phi\colon X\to Y$ be a bijection such that 
    \[
    \delta=\widehat{d}_{\mathfrak{N}}(X,Y)=\frac{1}{2} 
\max_{x,x'\in X}\left| \omega_X(x,x')-\omega_Y(\phi(x),\phi(x'))\right|\, .
    \]
    Then, for every $x,x'\in X$, we have 
    \[
    \omega_Y(\phi(x),\phi(x'))\leq \omega_X(x,x')+2\delta\,.
    \]
    Let $s\in \mathbb{R}$ and let $\sigma\in \mathcal{VR}^{s+2\delta}(X)$. Then $\phi(\sigma)$ is a simplex of $\mathcal{VR}^{s}(Y)$ since for all $x,x'\in \sigma$ it holds that
    \[
    \omega_Y(\phi(x),\phi(x'))\leq \omega_X(x,x')+2\delta\leq -s-2\delta+2\delta=-s\,.
    \]
    That is, for each $s\in \mathbb{R}$, the map $\phi$ defines a simplicial morphism
    \begin{equation*}
        \phi_s\colon \mathcal{VR}^{s+2\delta}(X)\longrightarrow\mathcal{VR}^{s}(Y)\,.
    \end{equation*}
    Analogously, the inverse bijection $\psi\colon Y\to X$ induces simplicial morphisms
\begin{equation*}
        \psi_s\colon \mathcal{VR}^{s+2\delta}(Y)\longrightarrow\mathcal{VR}^{s}(X)\,.
    \end{equation*}
These sets of morphisms make the following diagrams commute: 
\begin{equation}\label{Eq:DiagramasVR_Interleaving}
\begin{gathered}
    \begin{tikzcd}[column sep=6ex, row sep=5ex]
        \mathcal{VR}^{s}(X) & \mathcal{VR}^t(X) \arrow[l,"i_{s}^{t}"']\\
        \mathcal{VR}^{s+2\delta}(Y) \arrow[u,"\psi_s"']& \mathcal{VR}^{t+2\delta}(Y)\arrow[u,"\psi_t"'] \arrow[l, "i_{s+2\delta}^{t+2\delta}"']
    \end{tikzcd}
    \qquad 
    \begin{tikzcd}[column sep=3ex, row sep=5ex]
        \mathcal{VR}^{s-2\delta}(X) & & \mathcal{VR}^{s+2\delta}(X) \arrow[ll,"i_{s-2\delta}^{s+2\delta}"'] \arrow[ld, "\phi_s"]\\
        & \mathcal{VR}^{s}(Y) \arrow[lu,"\psi_{s-2\delta}"]
    \end{tikzcd}
    \\[4ex] 
    \begin{tikzcd}[column sep=6ex, row sep=5ex]
        \mathcal{VR}^{s}(Y)  & \mathcal{VR}^{t}(Y)\arrow[l,"i_{s}^{t}"']\\
        \mathcal{VR}^{s+2\delta}(X) \arrow[u,"\phi_s"']& \mathcal{VR}^{t+2\delta}(X)\arrow[u,"\phi_t"']\arrow[l, "i_{s+2\delta}^{t+2\delta}"']
    \end{tikzcd}
    \qquad 
    \begin{tikzcd}[column sep=3ex, row sep=5ex]
        \mathcal{VR}^{s-2\delta}(Y) & & \mathcal{VR}^{s+2\delta}(Y)\arrow[ll,"i_{s-2\delta}^{s+2\delta}"'] \arrow[ld, "\psi_{s}"]\\
        & \mathcal{VR}^{s}(X) \arrow[lu,"\phi_{s-2\delta}"] 
    \end{tikzcd}
\end{gathered}
\end{equation}
where $i_{s}^t$, $i_{s-2\delta}^{s+2\delta}$, and $i_{s+2\delta}^{t+2\delta}$ are the natural inclusions.

    Since all the morphisms forming the diagrams in equation~\eqref{Eq:DiagramasVR_Interleaving} are injective, they induce commutative diagrams between the $q$-coskeletons. Applying the degree $n$ simplicial cohomology functor to these diagrams yields the following commutative diagrams:
\[
\begin{gathered}
    \begin{tikzcd}[column sep=8ex, row sep=6ex, labels={font=\tiny}, nodes={font=\small}]
        H^n\mathcal{VR}^{s}(X)^q   
        & H^n\mathcal{VR}^t(X)^q  \arrow[l,"H^n(i_{s}^t)"', <-]\\
        H^n\mathcal{VR}^{s+2\delta}(Y)^q \arrow[u,"H^n(\psi_s^q)"', <-]   
        & H^n\mathcal{VR}^{t+2\delta}(Y)^q \arrow[u,"H^n(\psi_t^q)"', <-] \arrow[l, "H^n(i_{s+2\delta}^{t+2\delta})" ', <-] 
    \end{tikzcd}
    \quad 
    \begin{tikzcd}[column sep=-3ex, row sep=6ex, labels={font=\tiny}, nodes={font=\small}]
        H^n\mathcal{VR}^{s-2\delta}(X)^q    
        & & H^n\mathcal{VR}^{s+2\delta}(X)^q \arrow[ld,"H^n(\phi_s^q)", <-] \arrow[ll,"H^n(i_{s-2\delta}^{s+2\delta})" ', <-]\\
        & H^n\mathcal{VR}^{s}(Y)^q  \arrow[ul, "H^n(\psi_{s-2\delta}^q)", <-] &
    \end{tikzcd}
    \\[3.5ex] 
    \begin{tikzcd}[column sep=8ex, row sep=6ex, labels={font=\tiny}, nodes={font=\small}]
        H^n\mathcal{VR}^{s}(Y)^q    
        & H^n\mathcal{VR}^t(Y)^q \arrow[l,"H^n(i_{s}^t)"', <-] \\
        H^n\mathcal{VR}^{s+2\delta}(X)^q \arrow[u,"H^n(\phi_s^q)"', <-]
        & H^n\mathcal{VR}^{t+2\delta}(X)^q \arrow[u,"H^n(\phi_t^q)"', <-] \arrow[l, "H^n(i_{s+2\delta}^{t+2\delta})" ', <-] 
    \end{tikzcd}
    \quad 
    \begin{tikzcd}[column sep=-3ex, row sep=6ex, labels={font=\tiny}, nodes={font=\small}]
        H^n\mathcal{VR}^{s-2\delta}(Y)^q    
        & & H^n\mathcal{VR}^{s+2\delta}(Y)^q \arrow[ll,"H^n(i_{s-2\delta}^{s+2\delta})" ', <-] \arrow[ld,"H^n(\psi_s^q)", <-] \\
        & H^n\mathcal{VR}^{s}(X)^q   \arrow[ul, "H^n(\phi_{s-2\delta}^q)", <-]&
    \end{tikzcd}
\end{gathered}
\]
    That is, the persistence modules $H^n\mathcal{VR}^\bullet(X)^q$ and $H^n\mathcal{VR}^\bullet(Y)^q$ are $2\delta$-interleaved (Definition~\ref{Def:Interleaving}). By the Isometry Theorem relating the bottleneck distance and the interleaving distance (Theorem~\ref{Thm:isometria}), we conclude that 
\[
d_B\left(H^n\mathcal{VR}^\bullet(X)^q,H^n\mathcal{VR}^\bullet(Y)^q\right)\leq 2\,\delta=2\,\widehat{d}_{\mathfrak{N}}(X,Y)\,.
\]

\subsubsection{Proof of Theorem~\ref{Thm:EstabilidadPersistenciaGrosor}}\label{proof:Thm-EstabilidadPersistenciaGrosor}
Let $\delta$, $\phi$, and $\psi$ be as in the proof of Theorem~\ref{Thm:EstabilidadGruesos}. Following the notation from that proof, for any $s<t\in \mathbb{R}$, the commutativity of the diagrams in~\eqref{Eq:DiagramasVR_Interleaving} ensures that the following diagrams also commute:
\[
\begin{tikzcd}[row sep=4ex, column sep=2.75em]
    H^n\mathcal{VR}^s(X)^q 
        \arrow[rr, "H^n(i_s^t)"] 
        \arrow[dd, "\varphi^{n,s,q}_X"',<-] 
        \arrow[rd, "H^n(\psi_s^q)"] 
    & & 
    H^n\mathcal{VR}^t(X)^q
        \arrow[dd, "\varphi^{n,t,q}_X"' near start,<-] 
        \arrow[rd, "H^n(\psi_{t}^q)"] \\
    & 
    H^n\mathcal{VR}^{s+2\delta}(Y)^q 
        \arrow[rr, "H^n(i_{s+2\delta}^{t+2\delta})"' near start, crossing over] 
    & & 
    H^n\mathcal{VR}^{t+2\delta}(Y)^q 
        \arrow[dd, "\varphi^{n,t+2\delta,q}_Y", <-] \\
    H^n\mathcal{VR}^{s}(X)
        \arrow[rr, "H^n(i_{s}^t)" near start] 
        \arrow[rd, "H^n(\psi_{s})"] 
    & & 
    H^n\mathcal{VR}^{t}(X)
        \arrow[rd, "H^n(\psi_t)"] \\
    & 
    H^n\mathcal{VR}^{s+2\delta}(Y)
        \arrow[rr, "H^n(i_{s+2\delta}^{t+2\delta})", crossing over] 
        \arrow[from=uu, "\varphi^{n,s+2\delta,q}_Y" near start, crossing over, <-] 
    && 
    H^n\mathcal{VR}^{t+2\delta}(Y)
\end{tikzcd}
\]
\[
\begin{tikzcd}[row sep=4ex, column sep=2.75em]
    H^n\mathcal{VR}^{s-2\delta}(X)^q
        \arrow[rr, "H^n(i_{s-2\delta}^{s+2\delta})"'] 
        \arrow[dd, "\varphi^{n,s-2\delta,q}_X"',<-] 
        \arrow[rd, "H^n(\psi_{s-2\delta}^q)"'] 
    & & 
    H^n\mathcal{VR}^{s+2\delta}(X)^q 
        \arrow[dd, "\varphi^{n,s+2\delta,q}_X",<-]  \\
    & 
    H^n\mathcal{VR}^{s}(Y)^q\arrow[ur, "H^n(\phi_s^q)"'] \\
    H^n\mathcal{VR}^{s-2\delta}(X)
        \arrow[rr, "H^n(i_{s-2\delta}^{s+2\delta})" near start] 
        \arrow[rd, "H^n(\psi_{s-2\delta})"'] 
    & & 
    H^n\mathcal{VR}^{s+2\delta}(X) \\
    & 
    H^n\mathcal{VR}^{s}(Y) 
        \arrow[from=uu, "\varphi^{n,s,q}_Y" near start, crossing over,<-]
        \arrow[ur,"H^n(\phi_s)"'] 
\end{tikzcd}
\]
\[
\begin{tikzcd}[row sep=4ex, column sep=2.75em]
    H^n\mathcal{VR}^s(Y)^q
        \arrow[rr, "H^n(i_s^t)"] 
        \arrow[dd, "\varphi^{n,s,q}_Y"',<-] 
        \arrow[rd, "H^n(\phi_s^q)"] 
    & & 
    H^n\mathcal{VR}^t(Y)^q
        \arrow[dd, "\varphi^{n,t,q}_Y"' near start,<-] 
        \arrow[rd, "H^n(\phi_{t}^q)"] \\
    & 
    H^n\mathcal{VR}^{s+2\delta}(X)^q 
        \arrow[rr, "H^n(i_{s+2\delta}^{t+2\delta})"' near start, crossing over] 
    & & 
    H^n\mathcal{VR}^{t+2\delta}(X)^q 
        \arrow[dd, "\varphi^{n,t+2\delta,q}_X",<-] \\
    H^n\mathcal{VR}^{s}(Y)
        \arrow[rr, "H^n(i_{s}^t)" near start] 
        \arrow[rd, "H^n(\phi_{s})"] 
    & & 
    H^n\mathcal{VR}^{t}(Y)
        \arrow[rd, "H^n(\phi_t)"] \\
    & 
    H^n\mathcal{VR}^{s+2\delta}(X)
        \arrow[rr, "H^n(i_{s+2\delta}^{t+2\delta})", crossing over] 
        \arrow[from=uu, "\varphi^{n,s+2\delta,q}_X" near start, crossing over, <-] 
    && 
    H^n\mathcal{VR}^{t+2\delta}(X)
\end{tikzcd}
\]
\[
\begin{tikzcd}[row sep=4ex, column sep=2.75em]
    H^n\mathcal{VR}^{s-2\delta}(Y)^q 
        \arrow[rr, "H^n(i_{s-2\delta}^{s+2\delta})"] 
        \arrow[dd, "\varphi^{n,s-2\delta,q}_Y"',<-] 
        \arrow[rd, "H^n(\phi_{s-2\delta}^q)"'] 
    & & 
    H^n\mathcal{VR}^{s+2\delta}(Y)^q
        \arrow[dd, "\varphi^{n,s+2\delta,q}_Y",<-]  \\
    & 
    H^n\mathcal{VR}^{s}(X)^q\arrow[ur, "H^n(\psi_s^q)"'] \\
    H^n\mathcal{VR}^{s-2\delta}(Y) 
        \arrow[rr, "H^n(i_{s-2\delta}^{s+2\delta})" near start] 
        \arrow[rd, "H^n(\phi_{s-2\delta})"'] 
    & & 
    H^n\mathcal{VR}^{s+2\delta}(Y) \\
    & 
    H^n\mathcal{VR}^{s}(X) 
        \arrow[from=uu, "\varphi^{n,s,q}_X" near start, crossing over,<-]
        \arrow[ur,"H^n(\psi_s)"'] 
\end{tikzcd}
\]
From these diagrams, we obtain $2\delta$-interleaving morphisms between the persistence modules of the images $\img\varphi_X^{n,\bullet,q}$ and $\img\varphi^{n,\bullet,q}_Y$, the kernels $\ker\varphi^{n,\bullet,q}_X$ and $\ker\varphi^{n,\bullet,q}_Y$, and the cokernels $\coker\varphi^{n,\bullet,q}_X$ and $\coker\varphi^{n,\bullet,q}_Y$. Since the bottleneck distance and the interleaving distance coincide by the Isometry Theorem (Theorem~\ref{Thm:isometria}), the result follows.

\subsubsection{Proof of Theorem~\ref{Thm:EstabilidadCohesion}}\label{proof:Thm-EstabilidadCohesion}
    Following the notation in the proof of Theorem~\ref{Thm:EstabilidadGruesos}, the simplicial morphisms forming the diagrams in equation~\eqref{Eq:DiagramasVR_Interleaving} induce commutative diagrams between the corresponding face posets, where all induced maps are injective and order-preserving. Due to this injectivity, these maps restrict to order-preserving maps between the $\mathbf{h}$-face posets, making the following diagrams commute:
    \[
\begin{gathered}
    \begin{tikzcd}[column sep=6ex, row sep=4ex]
        P_{\mathcal{VR}^{s}(X)}^{\mathbf{h}} \arrow[r,"i_{s}^t",<-] \arrow[d,"\psi_s^{\mathbf{h}}",<-] 
        & P_{\mathcal{VR}^t(X)}^{\mathbf{h}} \arrow[d,"\psi_t^{\mathbf{h}}",<-] \\
        P_{\mathcal{VR}^{s+2\delta}(Y)}^{\mathbf{h}} \arrow[r, "i_{s+2\delta}^{t+2\delta}",<-] 
        & P_{\mathcal{VR}^{t+2\delta}(Y)}^{\mathbf{h}}
    \end{tikzcd}
    \qquad 
    \begin{tikzcd}[column sep=3ex, row sep=4ex]
        P_{\mathcal{VR}^{s-2\delta}(X)}^{\mathbf{h}} \arrow[rd, "\psi_{s-2\delta}^{\mathbf{h}}"',<-] \arrow[rr,"i_{s-2\delta}^{s+2\delta}",<-] 
        & & P_{\mathcal{VR}^{s+2\delta}(X)}^{\mathbf{h}} \\
        & P_{\mathcal{VR}^{s}(Y)}^{\mathbf{h}} \arrow[ru,"\phi_s^{\mathbf{h}}"',<-] &
    \end{tikzcd}
    \\[3ex] 
    \begin{tikzcd}[column sep=6ex, row sep=4ex]
        P_{\mathcal{VR}^{s}(Y)}^{\mathbf{h}} \arrow[r,"i_{s}^t",<-] \arrow[d,"\phi_s^{\mathbf{h}}",<-] 
        & P_{\mathcal{VR}^t(Y)}^{\mathbf{h}} \arrow[d,"\phi_t^{\mathbf{h}}",<-] \\
        P_{\mathcal{VR}^{s+2\delta}(X)}^{\mathbf{h}} \arrow[r, "i_{s+2\delta}^{t+2\delta}",<-] 
        & P_{\mathcal{VR}^{t+2\delta}(X)}^{\mathbf{h}}
    \end{tikzcd}
    \qquad 
    \begin{tikzcd}[column sep=3ex, row sep=4ex]
        P_{\mathcal{VR}^{s-2\delta}(Y)}^{\mathbf{h}} \arrow[rd, "\phi_{s-2\delta}^{\mathbf{h}}"',<-] \arrow[rr,"i_{s-2\delta}^{s+2\delta}",<-] 
        & & P_{\mathcal{VR}^{s+2\delta}(Y)}^{\mathbf{h}} \\
        & P_{\mathcal{VR}^{s}(X)}^{\mathbf{h}} \arrow[ru,"\psi_s^{\mathbf{h}}"',<-] &
    \end{tikzcd}
\end{gathered}
\]
    Taking poset cohomology with coefficients in $\Bbbk$ on the previous diagrams yields $2\delta$-interleaving morphisms between the persistence modules $H^nP_{\mathcal{VR}^\bullet(X)}^{\mathbf{h}}$ and $H^nP_{\mathcal{VR}^\bullet(Y)}^{\mathbf{h}}$. By the Isometry Theorem (Theorem~\ref{Thm:isometria}), we conclude that 
\[
d_{B}\big(H^n P_{\mathcal{VR}^\bullet(X)}^{\mathbf{h}},H^nP_{\mathcal{VR}^\bullet(Y)}^{\mathbf{h}}\big)\leq 2\,\widehat{d}_\mathfrak{N}(X,Y)\,.
\]

\subsubsection{Proof of Theorem~\ref{Thm:EstabilidadPersistenciaCohesion}}\label{proof:Thm_EstabilidadPersistenciaCohesion}
    The proof is analogous to that of Theorem~\ref{Thm:EstabilidadPersistenciaGrosor}, replacing the simplicial cohomology of the $q$-coskeleta with the poset cohomology of the $\mathbf{h}$-face posets. The required $2\delta$-interleaving morphisms arise from the commutativity of the diagrams induced on the face posets by equation~\eqref{Eq:DiagramasVR_Interleaving}, as explained in the proof of Theorem~\ref{Thm:EstabilidadCohesion}.
\end{document}